\documentclass[10pt]{amsart}
\usepackage{hyperref}
\usepackage{esint}
\usepackage[margin=1in]{geometry}
\usepackage[latin1]{inputenc}
\usepackage{amsfonts}
\usepackage{enumitem}
\usepackage{color}
\usepackage{bm}
\usepackage{textcomp}
\usepackage [english]{babel}
\usepackage{graphicx,dsfont}
\usepackage{amssymb,amsmath,amsthm}
\usepackage{changepage}

\newtheorem{theorem}{Theorem}[section]
\newtheorem{definition}{Definition}[section]
\newtheorem{proposition}{Proposition}[section]
\newtheorem{lemma}{Lemma}[section]
\newtheorem{corollary}{Corollary}[section]

\newtheorem{remark}{Remark}[section]
\theoremstyle{plain} 
\let\div\undefined
\DeclareMathOperator{\div}{div}

\catcode`\@=11
\makeatletter

\@addtoreset{equation}{section}

\begin{document}

\title[Global gradient estimates and second-order regularity]{
Global gradient estimates for solutions of parabolic equations with nonstandard growth}
%}

\author{Rakesh Arora}
\address{Department of Mathematical Sciences, Indian Institute of Technology (IIT-BHU), Varanasi, 221005, Uttar Pradesh, India}
\email{rakesh.mat@itbhu.ac.in} %, arora.npde@gmail.com}
\thanks{The first author acknowledges the support of the SERB Research Grant SRG/2023/000308, India, and Seed grant IIT(BHU)/DMS/2023-24/493.}

%    author two information
\author{Sergey Shmarev}
\address{Department of Mathematics, University of Oviedo, c/Federico Garc\'{i}a Lorca, Oviedo, 33007, Asturias, Spain}
\email[Corresponding author]{shmarev@uniovi.es}
\thanks{The second author acknowledges the support of the Research Grant MCI-21-PID2020-116287GB-I00, Spain}

%\today

\keywords{Nonlinear parabolic equation, Nonstandard growth, Global Gradient estimates, Second-order regularity}

\subjclass[2020]{35K65, 35K67, 35B65, 35K55, 35K99}

\begin{abstract}
We study how the smoothness of the initial datum and the free term affect the global regularity properties of solutions to the Dirichlet problem for the class of parabolic equations of $p(x,t)$-Laplace type %with nonlinear sources depending on the solution and its gradient:

\[
u_t-\Delta_{p(\cdot)}u=f(z)+F(z,u,\nabla u),\quad z=(x,t)\in Q_T=\Omega\times (0,T),
\]
with the nonlinear source $F(z,u,\nabla u)=a(z)|u|^{q(z)-2}u+|\nabla u|^{s(z)-2}(\vec c,\nabla u)$.
It is proven the existence of a solution such that if $|\nabla u(x,0)|\in L^r(\Omega)$ for some $r\geq \max\{2,\max p(z)\}$, then the gradient preserves the initial order of integrability in time, gains global higher integrability, and the solution acquires the second-order regularity in the following sense:
\[
\text{$|\nabla u(x,t)|\in L^r(\Omega)$ for a.e. $t \in (0,T)$}, \qquad \text{$|\nabla u|^{p(z)+\rho+r-2} \in L^1(Q_T)$ for any $\rho \in \left(0, \frac{4}{N+2}\right)$},
\]
and
\[
|\nabla u|^{\frac{p(z)+r}{2}-2}\nabla u\in L^2(0,T;W^{1,2}(\Omega))^N.
\]
The exponent $r$ is arbitrary and independent of $p(z)$ if $f\in L^{N+2}(Q_T)$, while for $f\in L^\sigma(Q_T)$ with $\sigma \in (2,N+2)$ the exponent $r$ belongs to a bounded interval whose endpoints are defined by $\max p(z)$, $\min p(z)$, $N$, and $\sigma$. An integration by parts formula is also proven, which is of independent interest.
\end{abstract}

\maketitle

% \linenumbers

\tableofcontents

\section{Introduction}
We study the global regularity properties of solutions to the problem

\begin{equation}
\label{eq:main}
\begin{split}
& u_t-\operatorname{div}\left(|\nabla u|^{p(z)-2}\nabla u\right)=F(z,u,\nabla u)+f(z)\quad \text{in $Q_T=\Omega\times (0,T)$},
\\
& \text{$u=0$ on $\partial\Omega\times (0,T)$},
\\
& \text{$u(x,0)=u_0(x)$ in $\Omega$}
\end{split}
\end{equation}
with the nonlinear source of the form

\[
F(z,u,\nabla u)=a(z)|u|^{q(z)-2}u+ |\nabla u|^{s(z)-2}(\vec c(z), \nabla u).
\]
The coefficients $a$, $\vec{c}$, the exponents $p$, $q$, $s$ and the free term $f$ are given functions of the independent variables $z=(x,t)\in Q_T$. It is assumed that $\Omega\subset \mathbb{R}^N$, $N\geq 2$, is a bounded domain with the boundary  $\partial\Omega\in C^{2+\alpha}$, $\alpha\in (0,1)$.

Equation \eqref{eq:main} belongs to the class of equations with nonstandard growth because the exponents of nonlinearity $p$, $q$, and $s$ are not constants but functions of the independent variables, which causes a gap between the coercivity and growth conditions in the principal part of the equation. Nonlinearities of this kind are found in various real-world applications, such as mathematical models of electro-rheological or thermo-rheological fluids, as well as in image processing.

When $|\nabla u|=0$, equation \eqref{eq:main} degenerates if $p(z)>2$, or becomes singular if $p(z)<2$. Therefore, the solution needs to be understood in a weak sense. The study of the local regularity properties of a weak solution to a nonlinear parabolic equation is often practically independent of the method used to prove its existence. These results apply to every solution from a specific class and depend only on the nonlinear structure of the equation. In contrast, global estimates or estimates that continue to hold up to the boundary of the problem domain require information about the regularity of the domain and the data. The questions of regularity of weak solutions to nonlinear evolution equations involving the $p(z)$-Laplace, or the regularized nondegenerate $p(z)$-Laplace operators,

\[
\Delta_{p(\cdot)}u=\operatorname{div}\left(|\nabla u|^{p(z)-2}\nabla u\right),\qquad \Delta^{(\mu)}_{p(\cdot)}u=\operatorname{div}\left(\left((\mu+|\nabla u|^2\right)^{\frac{p(z)-2}{2}}\nabla u\right),\quad \mu\not=0,
\]
have been studied by many researchers.  The weak solutions of the equation

\begin{equation}
\label{eq:p-Laplace}
u_t=\Delta_{p(\cdot )} u+f,\qquad \min p(z)>1,
\end{equation}
or systems of equations of a similar structure possess the property of improved integrability of the gradient: instead of the inclusion $|\nabla u|^{p(z)}\in L^{1}(Q_T)$ prompted by the equation, for the solutions of equation \eqref{eq:p-Laplace}

\[
\int_{S}|\nabla u|^{p(z)+\delta}\,dz\leq C
\]
in the arbitrary strictly interior cylinder $S\Subset Q_T$, with a constant $\delta>0$, which depends on the distance from $S$ to the parabolic boundary of $Q_T$. For variable exponents $p(z)$, this property was found in \cite{Ant-Zhikov-2005} for $p(z)\geq 2$ and in \cite{Zhikov-Past-2010} in the range $\frac{2N}{N+2}<p(z)\leq p^+<\infty$ for the exponents $p(z)$ with the logarithmic modulus of continuity. The global version of this inequality was proven \cite{Ar-Sh-2021,Ar-Sh-RACSAM-2023} for Lipschitz-continuous exponents $p(\cdot)$ and $\partial\Omega\in C^2$:

\begin{equation}
\label{eq:high-int-intr}
\int_{Q_T}|\nabla u|^{p(z)+\delta}\,dz\leq C\quad \text{with any $\delta \in \left(0,\frac{4}{N+2}\right)$}
\end{equation}
and a constant $C$ depending on $N$, $\max p$, $\min p$, $\delta$ and $\|u\|_{L^{\infty}(0,T;L^2(\Omega))}$. We also refer to the recent work  \cite{Hasto-OK-2021} for the proof of local higher integrability of the gradients in solutions of systems of parabolic equations with Orlicz growth, and to \cite{KKM-2023} for the systems of double-phase parabolic equations.

Systems of parabolic equations

\begin{equation}
\label{eq:intr-system}
u^{i}_t-\sum_{\alpha=1}^n\left[\mathbf{a}(|\nabla u|)u_{x_\alpha}^{i}\right]_{x_{\alpha}}=b^{i},\quad i=\overline{1,N},\;\;n\geq 2,
\end{equation}
in a cylinder $\Omega\times (0,T)$ with bounded and convex $\Omega$ are studied in \cite{Bogelein-2021}. It is assumed that the function $\mathbf{a}(\cdot)$ satisfies $\mathbf{a}(s)+s\mathbf{a}'(s)\approx s^{p-2}$ with constant $p>\frac{2N}{N+2}$, and that $b^{i}\in L^\sigma(Q_T)$ with $\sigma>N+2$. It is shown that on every time interval $(\epsilon, T)$, $\epsilon>0$, the gradient of the solution is bounded up to the lateral boundary where the homogeneous Dirichlet condition is posed. In the paper \cite{DeFilippis-2020}, local $L^\infty$ gradient bounds are derived for a class of parabolic equations with $(p,q)$-growth. The results of \cite{DeFilippis-2020} apply to \eqref{eq:p-Laplace} with $f=0$, as well as certain double-phase parabolic equations. The optimal conditions for local boundedness of the gradient of solutions to systems of parabolic equations of the structure \eqref{eq:p-Laplace} are found in \cite{Kuusi-Mignione-2012} in the framework of borderline rearrangement invariant function spaces
of Lorentz type under the assumption $f\in L(N + 2,1)$, see also \cite{Kuusi-Mingione-2012-NA, Kuusi-Mingione-2013-1}. In the case $f \in L^{\sigma}(Q_T)$ with $\sigma \leq N+2$, the question of global (and even local) gradient regularity remains open.

Further results concern local H\"older continuity of the gradient. We refer to \cite{Bogelein-2022} for systems \eqref{eq:intr-system} with $b=0$, and \cite{OK-Scilla-Stroffolini-2024} for systems of parabolic equations with the Uhlenbeck structure, see also further references therein. For the solutions of problem  \eqref{eq:main} with nonlinear sources subject to appropriate growth conditions and $f\in L^{\infty}(Q_T)$, the H\"older continuity of $\nabla u$ in every cylinder $\Omega\times (\epsilon,T)$, $\epsilon>0$, with $\partial\Omega\in C^{1+\alpha}$ is proven in \cite{Ding-Zhang-Zhou-2020}. For the viscosity solutions of the normalized $p$-Laplace equation, local H\"older continuity of the spatial gradient is proven in \cite{Andrade-Santos-2022-viscosity-sol} for $p=const$, and in \cite{Imb-Jin-Silv-2019-viscosity-sol} for the differentiable variable exponent $p(x,t)$.

The results on the second-order regularity of solutions to equations and systems that involve the $p$-Laplace operator are usually formulated in terms of the inclusions

\[
|\nabla u|^{\lambda}\nabla u\in L^2(0,T;W^{1,2}(\Omega))^N\qquad \text{or} \quad L^2_{loc}(0,T;W^{1,2}_{loc}(\Omega))^N
\]
with an exponent $\lambda$ depending on $p$ and $N$. We refer here to \cite{Seregin-Acerbi-Mignione-2004,Duzaar-Mignione-Steffen-2011,Feng-Parviainen-Sarsa-2023,Feng-Parviainen-Sarsa-2022,DeFilippis-2020} for results on the local second-order regularity of weak and viscous solutions to equations of the type \eqref{eq:p-Laplace} with constant or variable $p$, and to \cite{Berselli-Ruzicka-2022} for global regularity results, see also references therein.

The global second-order regularity of weak solutions to the Dirichlet problem for equation \eqref{eq:p-Laplace} or \eqref{eq:main} with different choices of the source terms and the exponents of nonlinearity was studied in \cite{Cianchi-Maz'ya-2019-1,Ar-Sh-2021,Ar-Sh-2024}. Paper \cite{Cianchi-Maz'ya-2019-1} deals with the \textit{approximable solutions} of the homogeneous Dirichlet problem for equation \eqref{eq:p-Laplace} with $p=\text{const.}>1$. It is proven that for such solutions, which are the limits of solutions to the problem with
smooth data,

\[
\begin{split}
& u_t\in L^2(Q_T), \qquad |\nabla u|^{p-2}\nabla u\in L^2(0,T;W^{1,2}(\Omega))^N,
\\
& \|u_t\|_{L^2(Q_T)}+\||\nabla u|^{p-2}\nabla u\|_{L^2(0,T;W^{1,2}(\Omega))}\leq C
\end{split}
\]
with a constant $C$ depending only on $\|u_0\|_{W^{1,p}_0(\Omega)}$ and $\|f\|_{L^2(Q_T)}$. This global result is optimal in the sense that in the special case $u_0=0$ the $L^2(0,T;W^{1,2}(\Omega))$-norm of the flux function $|\nabla u|^{p-2}\nabla u$ is bounded on both sides by $C_i\|f\|_{L^2(Q_T)}$ with constants $C_i$ independent of the solution $u$. The regularity results of \cite{Cianchi-Maz'ya-2019-1} are obtained under very weak assumptions on the smoothness of $\partial\Omega$. The pointwise convergence of the gradients of solutions to the approximate problems to the gradient of the approximable solution plays a crucial role in the proof.
In \cite{Ar-Sh-2021}, the solution of the homogeneous Dirichlet problem for the evolution $p(z)$-Laplace equation is obtained as the limit of the sequence of Galerkin's approximations. As distinguished from \cite{Ar-Sh-2021}, the scheme of construction of a solution proposed in \cite{Ar-Sh-2024} consists of approximating both the data and the nonlinearities in problem \eqref{eq:main} by smooth functions, obtaining a family of classical solutions to the approximate problems, and proving that this family contains a sequence that converges to a solution of the original problem. For the solutions of problem \eqref{eq:main} with $F\equiv 0$, $f\in L^2(0,T;W^{1,2}_0(\Omega))$, and $u_0\in W^{1,q(\cdot)}_0(\Omega)$ with $q(x)=\max\{2,p(x,0)\}$,  the inclusions

\[
u_t\in L^2(Q_T),\quad D_{x_i}\left(|\nabla u|^{\frac{p(z)-2}{2}}D_{x_j} u\right)\in L^2(Q_T),\;\;i,j=\overline{1,N},
\]
and the estimates on the corresponding norms are obtained in \cite{Ar-Sh-2021} for the Lipschitz-continuous exponent $p(z)$ in the range $p(z)>\max\left\{\frac{2N}{N+2},\frac{6}{5}\right\}$. The restiction $p(z)>\frac{6}{5}$ if $N=2$ was removed in \cite{Ar-Sh-RACSAM-2023}. It is proved in \cite{Ar-Sh-2024} that problem \eqref{eq:main} with $p(z)>\frac{2(N+1)}{N+2}$, $u_0\in W^{1,p(\cdot,0)}_0(\Omega)$, and $f\in L^2(Q_T)$ admits a solution $u$ with the improved properties

\begin{equation}
\label{eq:approx-L-2}
\begin{split}
& |\nabla u|^{2(p(z)-1)+\delta}\in L^1(Q_T) \quad \text{with any $\delta \in \left(0,\frac{4}{N+2}\right)$},
\\
& |\nabla u|^{p(z)-2}\nabla u\in L^2(0,T;W^{1,2}(\Omega)).
\end{split}
\end{equation}
Properties \eqref{eq:approx-L-2} follow from the uniform estimates on the classical solutions of the regularized problems and the fact that the gradients of solutions to the regularized problems converge pointwise to the gradient of the limit problem. The proof of the a.e. convergence of the gradients relies on the property  \eqref{eq:high-int-intr} of improved integrability of the gradient. These second-order regularity results offer a natural continuation of the results of \cite{Cianchi-Maz'ya-2019-1} to the case of variable exponents $p(z)$. Still, the approaches of \cite{Cianchi-Maz'ya-2019-1} and \cite{Ar-Sh-2024} are different.

To the best of our knowledge, even in the case of constant $p>1$, no results are available concerning the effect of the regularity of the initial data and nonlinear source term on the global integrability of the gradient and the second-order regularity of solutions to the Dirichlet problem for equation \eqref{eq:p-Laplace}.

In the present work, we address the global regularity properties of the gradients in solutions to problem \eqref{eq:main}. We are interested in the following question:

\begin{quote}
\centering
{\it how does the regularity of the initial function and the free term affect the global properties of the gradient of a solution to problem \eqref{eq:main}}?
\end{quote}

We show that under quite general conditions on the exponents of nonlinearity and the assumption that $f$ and $u_0$ satisfy one of the conditions

\begin{equation}
\label{eq:intr-0}
 \begin{split}
 {\rm (a)}& \qquad \text{$f\in L^{N+2}(Q_T)$},\qquad \text{$|\nabla u_0|\in L^r(\Omega) \quad \text{with} \quad r\geq \max\left\{2,\max_{\overline{Q}_T}p(z)\right\}:= R_p$},
 \\
  {\rm (b)}& \qquad \text{$f\in L^\sigma(Q_T)$ with $\sigma \in (2,N+2)$}, \qquad \text{$|\nabla u_0|\in L^r(\Omega)$ with $R_p \leq r  < R_\sigma$},
 \end{split}
\end{equation}
where $R_\sigma$ is a finite number defined by $\max_{\overline{Q}_T}p(z)$, $\min_{\overline{Q}_T}p(z)$, $N$, and $\sigma$,
problem \eqref{eq:main} admits a solution whose gradient preserves the initial order of integrability in time, gains global higher integrability, and acquires the second-order regularity in the following sense:

\begin{equation}
\label{eq:intr-1}
\begin{split}
\text{(a)} & \quad |\nabla u(x,t)| \in L^\infty(0,T;L^r(\Omega)),
\\
\text{(b)} & \quad |\nabla u|^{p(z)+\rho+r-2} \in {L^{1}(Q_T)} \quad \text{for any} \ \rho \in \left(0, \frac{4}{N+2}\right),
\\
\text{(c)} & \quad  |\nabla u|^{\frac{p(z)+r-2}{2}}\in L^2(0,T;W^{1,2}(\Omega)), \quad \text{and} \quad  D_{x_i}\left(|\nabla u|^{\frac{p(z)+r}{2}-2}D_{x_j}u\right)\in L^2(Q_T),\quad i,j=\overline{1,N}.
\end{split}
\end{equation}
The conditions on the exponents $p$, $q$, $s$ and the coefficients $a$ and $\vec c$ will be specified later (see \eqref{eq:data}).

Two features are worth noting. Firstly, the provable order of integrability of $|\nabla u|$ is different in the cases \eqref{eq:intr-0} (a) and (b). In case \eqref{eq:intr-0} (a) $|\nabla u|$ maintains its initial order of integrability $r$ in time and gains global higher integrability independently of $p$, while in case \eqref{eq:intr-0} (b), where the order of integrability of $f$ is below $N+2$, the value of $r$ falls within an interval $[R_p, R_\sigma)$ whose endpoints depend on the properties of $p$, $N$ and the exponent $\sigma$. Furthermore, as $\sigma$ approaches $(N+2)^-$, the right endpoint $R_\sigma$ tends to $+\infty$ and as $\sigma$ approaches $2^+$, the properties \eqref{eq:intr-1} (a), (b) and (c) coincide with the results obtained in \cite{Ar-Sh-2024} for the case $\sigma=2$ and $r=p(\cdot)$. Secondly, when proving properties \eqref{eq:intr-1}, we do not distinguish between the degenerate or singular cases, {\it i.e.}, it is possible that $p(z) \geq 2$ on the part of the cylinder $Q_T$ and $p(z) < 2$ on the rest. The results are new even for the model equation \eqref{eq:p-Laplace} with constant $p$. The present work can also be seen as a continuation of \cite{Ar-Sh-2024}, as the method used to construct a solution to \eqref{eq:main} in that work furnishes the framework for proving further regularity properties \eqref{eq:intr-1}.

Organization of the paper. Section \ref{sec:prelim} consists of an outline of the associated function spaces, formulation of the regularized problem, and statements of the main results of this work. In Section \ref{sec:Hessian}, we prove new differential and integral inequalities that result from the multiplication of the regularized equation \eqref{eq:main} by the regularized $r$-Laplacian and integration by parts in $\Omega$. In the case of variable $p(z)$, these formulas include residual terms, which do not belong to the natural energy space prompted by the equation. The residual terms are estimated with the help of the property of higher integrability of the gradient that follows from special interpolation inequalities.
In Section \ref{sec:est-source}, we derive integral estimates that involve unbounded free source term $f$ and the nonlinear sources dependent on $u$ and $\nabla u$. The estimates are obtained for $f\in L^{\sigma}(Q_T)$ with $\sigma>2$ and are different in the cases $\sigma = N+2$ and $\sigma \in (2,N+2)$. The proofs of the main results are given in Section \ref{sec:proof-results}. Finally, in Appendix \ref{sec:intbyparts}, we present the proof of the integration by parts formula, which is of independent interest.

\section{Preliminaries, assumptions and results}
\label{sec:prelim}
In this section, we collect known facts from the theory of variable Lebesgue and Sobolev spaces where the solutions to problem \eqref{eq:main} belong to, pose the approximate problems and recall the existence and regularity results obtained in \cite{Ar-Sh-2024}. The main results of the current work are presented in Section \ref{sec:results}.

\subsection{The function spaces}
\label{subsec:spaces}

Let $\Omega$ be a bounded domain with the Lipschitz continuous boundary
$\partial \Omega$ and $p: \Omega \to [p^-, p^+] \subset (1,
\infty)$ be a measurable function. Let us define the functional (the modular)

\[
A_{p(\cdot)}(f)=
\int_{\Omega} |f(x)|^{p(x)} \,dx.
\]
The set
\[
L^{p(\cdot)}(\Omega) = \{f:\Omega
\to \mathbb{R}:\text{$f$ is measurable on $\Omega$}, A_{p(\cdot)}(f) <
\infty\}
\]
equipped with the Luxemburg norm
\[
\|f\|_{p(\cdot),\Omega}= \inf \left\{\lambda>0 :
A_{p(\cdot)}\left(\dfrac{f}{\lambda}\right) \leq 1\right\}
\]
is a reflexive and separable Banach space and $C_0^\infty(\Omega)$ (the set of smooth functions with compact support) is dense in $L^{p(\cdot)}(\Omega)$. By the definition of the norm

\begin{equation}
\label{i0}
\min\left\{\|f\|^{p^-}_{p(\cdot), \Omega}, \|f\|^{p^+}_{p(\cdot), \Omega}\right\}
\leq A_{p(\cdot)}(f) \leq \max\left\{\|f\|^{p^-}_{p(\cdot), \Omega}, \|f\|^{p^+}_{p(\cdot), \Omega}\right\}.
\end{equation}
The dual of $L^{p(\cdot)}(\Omega)$ is the space $L^{p'(\cdot)}(\Omega)$ with the conjugate exponent  $p'=\dfrac{p}{p-1}$. For $f \in L^{p(\cdot)}(\Omega)$ and $g \in L^{p'(\cdot)}(\Omega)$, the generalized
H\"older inequality holds:

\begin{equation}
\label{eq:gen-Holder}
\int_{\Omega} |fg| \leq
\left(\frac{1}{p^-} + \frac{1}{(p')^-} \right) \|f\|_{p(\cdot),
\Omega} \|g\|_{p'(.), \Omega} \leq 2 \|f\|_{p(\cdot), \Omega}
\|g\|_{p'(\cdot), \Omega}.
\end{equation}
Let $p_1, p_2$ be two bounded
measurable functions in $\Omega$ such that $1<p_1(x) \leq p_2(x)$ a.e. in $\Omega$.  Then $L^{p_1(\cdot)}(\Omega)$ is continuously
embedded in $L^{p_2(\cdot)}(\Omega)$ and

\[
\|u\|_{p_1(\cdot), \Omega} \leq C(|\Omega|,
p_1^\pm, p_2^\pm) \|u\|_{p_2(\cdot), \Omega}, \qquad \forall \,u \in
L^{p_2(\cdot)}(\Omega).
\]
The variable Sobolev space $W_0^{1,p(\cdot)}(\Omega)$ is defined as the set of functions

\[
W_0^{1,p(\cdot)}(\Omega)= \{u: \Omega \mapsto
\mathbb{R}\ |\  u \in L^{p(\cdot)}(\Omega) \cap W_0^{1,1}(\Omega),
|\nabla u| \in L^{p(\cdot)}(\Omega) \}
\]
equipped with the norm

\[
\|u\|_{W_0^{1,p(\cdot)}(\Omega)}= \|u\|_{p(\cdot),\Omega} +
\|\nabla u\|_{p(\cdot), \Omega}.
\]
Under the assumptions $p\in C^{0}(\overline \Omega)$ and $\partial\Omega$ is Lipschitz-continuous, the
Poincar\'{e} inequality holds: there is a finite constant $C>0$ such that

\[
\|v\|_{p(\cdot),\Omega}\leq C\|\nabla v\|_{p(\cdot),\Omega}\qquad \forall \ v\in W_0^{1,p(\cdot)}(\Omega).
\]
It follows that $\|\nabla v\|_{p(\cdot),\Omega}$ is an equivalent norm of
$W^{1,p(\cdot)}_0(\Omega)$.

It is known that $C_0^{\infty}(\Omega)$ is dense in
$W_0^{1,p(\cdot)}(\Omega)$,
%and the Poincar\'e inequality holds
provided $p
\in C_{{\rm log}}(\overline{\Omega})$, {\it i.e.}, $p$ is continuous in $\overline{\Omega}$ with the logarithmic modulus of continuity:

\begin{equation}
\label{eq:log-cont}
|p(x_1)-p(x_2)| \leq \omega(|x_1-x_2|),
\end{equation}
where $\omega(\cdot)$ is a nonnegative function satisfying the condition
\[
\limsup_{\tau \to 0^+} \omega(\tau) \ln
\left(\frac{1}{\tau}\right)= C< \infty.
\]
For the study of parabolic problem \eqref{eq:main}, we need the spaces of functions depending on $(x,t)\in Q_T$:
\[
\begin{split}
 & \mathbf{V}_t(\Omega) = \{u: \Omega \mapsto \mathbb{R}\ |\  u \in L^2(\Omega)
\cap W_0^{1,1}(\Omega), |\nabla u|^{p(x,t)} \in L^{1}(\Omega)
\},\quad t\in (0,T),
\\
& \mathbf{W}_{p(\cdot)}(Q_T) = \{u : (0,T) \mapsto \mathbf{V}_t(\Omega) \ |\ u \in L^2(Q_T), |\nabla
u|^{p(x,t)} \in L^1(Q_T)\}.
\end{split}
\]

By $C^{0,1}(\overline Q_T)$ we denote the space of Lipschitz-continuous functions: a function $\phi$ is Lipschitz-continuous on $\overline Q_T$ if there exists a finite constant $L\geq 0$ such that

\[
|\phi(x,t)-\phi(y,\tau)|\leq Ld\left((x,t),(y,t)\right)\qquad \forall
(x,t),(y,\tau)\in \overline{Q}_T,
\]
where $d(\cdot,\cdot)$ is the Euclidean distance in $\mathbb{R}^{N+1}$.

By $C^{k+\alpha,\frac{k+\alpha}{2}}(Q_T)$, $k\in \mathbb{N}$, $\alpha\in (0,1)$, we denote the classical parabolic H\"older spaces.

\subsection{Assumptions and notation}
The solution of problem \eqref{eq:main} is understood in the following sense.
\begin{definition}
\label{def:1} A function $u$ is called strong solution of problem
\eqref{eq:main} if

\begin{enumerate}
\item $u\in C^0([0,T];L^2(\Omega))\cap
\mathbf{W}_{p(\cdot)}(Q_T)$, $u_t\in L^2(Q_T)$,

    \item for every $\phi\in \mathbf{W}_{p(\cdot)}(Q_T)$

    \begin{equation}
    \label{eq:def-1}
    \int_{Q_T}\left(u_t\phi+|\nabla u|^{p(z)-2}\nabla u\cdot \nabla \phi-f\phi -F(z,u,\nabla u)\phi\right)\,dz=0,
    \end{equation}

    \item $\|u(\cdot,t)-u_0(\cdot)\|_{2,\Omega}\to 0$ as $t\to 0^+$.
\end{enumerate}
\end{definition}

Let us assume that the data of problem \eqref{eq:main} satisfy the following conditions:

\begin{equation}
\label{eq:data-0}
\begin{split}
& \text{$\partial\Omega\in  C^{2+\alpha}$ with $\alpha\in (0,1)$},\quad q,s,a \in C^{0}(\overline Q_T), \quad \vec c \in C^0(\overline Q_T)^N,
\\
& \text{$p\in C^{0,1}(\overline{Q}_T)$ with the Lipschitz constant $L_p$},
\end{split}
\end{equation}
there exist positive constants $p^\pm$, $q^\pm$, $s^\pm$, $\mu$, $a^+$, $c_i^+$ such that $2\mu<p^--1$,
\begin{equation}
\label{eq:data}
\begin{split}
& p:\,\overline{Q}_T\mapsto [p^-,p^+],\qquad \frac{2(N+1)}{N+2}<p^-\leq
p^+<\infty,
\\
&
q:\,\overline{Q}_T\mapsto [q^-,q^+], \quad 1<q^-\leq q^+<\min\left\{p^-,1+p^-\frac{N+2}{2N}\right\}-2\mu,
\\
& s:\,\overline{Q}_T\mapsto [s^-,s^+], \quad 1<s^-\leq s^+<\infty,\quad s^+\leq p^--2\mu,
\\
& \max_{\overline{Q}_T}|a|=a^+,\quad
\max_{\overline{Q}_T}|c_i|=c_i^+.
\end{split}
\end{equation}

Under these assumptions, the solution to problem \eqref{eq:main} with $u_0\in W^{1,p(\cdot,0)}_0(\Omega)$ and $f\in L^2(Q_T)$ can be obtained as the pointwise limit of the family of classical solutions to the regularized problem \eqref{eq:main} (problem \eqref{eq:main-reg} below). The procedure of construction of strong solutions to the regularized  problems and the degenerate problem is presented in the next section.

\textbf{Notation.} By agreement, we denote $p_0(x)=p(x,0)$ and use the
abbreviations,

\[
D_jv=\dfrac{\partial v}{\partial x_j},\qquad |v_{xx}|^2=\sum_{i,j=1}^N\left(D_{ij}^2v\right)^2,\qquad \|v\|_{L^s(\Omega)}=\|v\|_{s,\Omega},\;\;s\geq 1.
\]

The notation $\partial \Omega\in C^{2+\alpha}$ (or $C^2$) means that for every
$x_0\in \partial\Omega$ there exists a ball $B_R(x_0)$ such that in the local
coordinate system $\{y_i\}$ centered at $x_0$ with $y_N$ pointing in the
direction of the exterior normal to $\partial\Omega$ at $x_0$, the set
$B_R(x_0)\cap \partial\Omega$ can be represented as the graph of a function
$\phi(y_1,\ldots,y_{N-1})\in C^{2+\alpha}$ (or $C^2$).

The symbol $C$ stands for constants that can be evaluated through the known quantities but whose exact values are unimportant. The value of $C$ may vary from line to line even inside the same formula.

\subsection{Regularized problem}

Consider the regularized problem

\begin{equation}
\label{eq:main-reg}
\begin{split}
& u_t-\operatorname{div}\left((\epsilon^2+|\nabla u|^{2})^{\frac{p(z)-2}{2}}\nabla u\right)=F_\epsilon(z,u,\nabla u)+f(z)\quad \text{in $Q_T$},
\\
& \text{$u=0$ on $\partial\Omega\times (0,T)$},
\\
& \text{$u(x,0)=u_0(x)$ in $\Omega$},\qquad \epsilon\in (0,1),
\end{split}
\end{equation}
where $f$ and $u_0$ satisfy \eqref{eq:intr-0} and the nonlinear source $F_\epsilon$ has the form

\begin{equation}
\label{eq:sources}
\begin{split}
F_\epsilon(z,u,\nabla u) &=F_{1\epsilon}(z,u)+ F_{2\epsilon}(z,\nabla u)
\\
&
\equiv a(z)(\epsilon^2+|u|^2)^{\frac{q(z)-2}{2}}u+ (\epsilon^2+|\nabla u|^2)^{\frac{s(z)-2}{2}}(\vec c(z),\nabla u).
\end{split}
\end{equation}
The strong solution of the regularized problem is understood in the sense of Definition \ref{def:1}. The data of problem \eqref{eq:main-reg} are smooth approximations of the data of problem \eqref{eq:main}. If the data of \eqref{eq:main} satisfy conditions \eqref{eq:intr-0}, \eqref{eq:data-0}, \eqref{eq:data} and $\sigma \geq 2$, there exist sequences $\{u_{0m}\}$, $\{f_m\}$ such that

\begin{equation}
\label{eq:data-reg-1}
\begin{split}
& u_{0m}\in C_0^{\infty}(\Omega),\qquad \text{$u_{0m}\to u_0$ in
$W^{1,r}_0(\Omega)$},
\\
& f_m\in C_0^{\infty}(Q_T),\qquad \text{$f_m\to f$ in $L^{\sigma}(Q_T)$},
\end{split}
\end{equation}
and the sequences of coefficients and exponents satisfying

\begin{equation}
\label{eq:data-reg-0}
\begin{split}
& p_k\in C^\infty (\overline Q_T), \quad \text{$p_k\nearrow p$ in $C^{0,1}(\overline Q_T)$},
\\
& s_k,q_k\in C^{\infty}(\overline Q_T), \quad \text{$q_k\to q$, $s_k\to s$ in $C^0(\overline Q_T)$},
\\
& a_k,c_{ik}\in C^{\infty}(\overline Q_T), \quad \text{$a_k\to a$, $c_{ik}\to c_i$ in $C^0(\overline Q_T)$}.
\end{split}
\end{equation}
We refer to \cite[Subsec.3.2]{Ar-Sh-2024} for the choice of these sequences. By agreement, we will use the notation $\textbf{data}$ for the set of functions $\{u_0,f,p,q,s,a,\vec c\}$ satisfying conditions \eqref{eq:intr-0}, \eqref{eq:data-0}, \eqref{eq:data}. The notation $\textbf{data}_m$ will be used for the sequence of smooth functions $\{u_{0m},f_m,p_m,q_m,s,a,\vec c_m\}$ whose elements approximate $\textbf{data}$ in the sense of \eqref{eq:data-reg-1}, \eqref{eq:data-reg-0}. Because of the choice of approximating sequence $\textbf{data}_m$, the norms of its elements converge to the corresponding norms of $\textbf{data}$,
\[
\|f_m\|_{\sigma,Q_T}=\|f\|_{\sigma,Q_T}+o(1),\qquad \|u_{0m}\|_{r,\Omega}=\|u_0\|_{r,\Omega}+o(1)\quad \text{as $m\to \infty$},
\]
and the same is true for the sequences of coefficients and exponents. Since $\|\textbf{data}_m\|$ tend to $\|\textbf{data}\|$, we may substitute $\|\textbf{data}_m\|$ by $1+\|\textbf{data}\|$ in the constants that appear in estimating the solutions of problem \eqref{eq:main-reg} and increase as $\|\textbf{data}_m\|$ grows.

The proof of existence of a strong solution to problem \eqref{eq:main} is given in \cite{Ar-Sh-2024}. It is proven first that problems \eqref{eq:main-reg} with smooth data $\textbf{data}_m$ admit classical solutions $u_{\epsilon m}$, and that the sequence $\{u_{\epsilon m}\}$ converges as $m\to \infty$ to a strong solution $u_\epsilon$ of problem \eqref{eq:main-reg}.

\begin{proposition}[\cite{Ar-Sh-2024},~Lemmas 3.3, 5.1]
\label{pro:existence-smooth-data}
Let $\partial\Omega\in C^{2+\alpha}$, $\alpha\in (0,1)$. Problem \eqref{eq:main-reg} with the data $\textbf{data}_m$ has a unique classical solution $u\equiv u_{\epsilon m}\in C^{2+\gamma,1+\frac{\gamma}{2}}(\overline Q_T)$ with some $\gamma\in (0,1)$, and $u(\cdot,t)\in C^3(\Omega)$ for every $t\in [0,T]$. Moreover, for every $m\in \mathbb{N}$

\begin{equation}
\label{eq:unif-reg}
\begin{split}
\|u_t\|^2_{2,Q_T} & + \sup_{(0,T)}\|u\|_{2,\Omega}^2 + \sup_{(0,T)}
 \int_{\Omega}(\epsilon^2 +|\nabla u|^2)^{\frac{p_m}{2}}\,dx +\int_{Q_T}(\epsilon^2+|\nabla u|^2)^{p_m-2}|u_{xx}|^2\,dz
 \\
 &
 \leq C\left(1+\int_\Omega |\nabla u_{0}|^{p(\cdot,0)}\,dx +\|f\|_{2,Q_T}^2\right)
\end{split}
\end{equation}
with an independent of $\epsilon$ and $m$ constant $C$. The solutions possess the property of higher integrability of the gradient: for every $\delta\in \left(0,\frac{4}{N+2}\right)$

\begin{equation}
\label{eq:high-integr-m}
\int_{Q_T} (\epsilon^2+|\nabla u|^2)^{p_m-2+\frac{\delta}{2}}|\nabla u|^2\,dx\leq C'
\end{equation}
with a constant $C'$ depending only on $\textbf{data}$ and $\delta$.
\end{proposition}

The solution $u_\epsilon$ of problem \eqref{eq:main-reg} corresponding to $\textbf{data}$ is the pointwise limit of the sequence $\{u_{\epsilon m}\}$ as $m\to \infty$. Moreover, $\nabla u_{\epsilon m}\to \nabla u_\epsilon$ in $L^{p(\cdot)}(Q_T)$ and a.e. in $Q_T$. Since $u_{\epsilon m}\in \mathbf{W}_{p_m}(Q_T)$ with possibly different exponents $p_m(z)$, the convergence of the sequence $\{u_{\epsilon m}\}$ in $\mathbf{W}_{p}(Q_T)$ with $p=\lim p_m$ requires a special justification. This fact follows from the property \eqref{eq:high-integr-m} of higher integrability of the gradients and the uniform convergence $p_m\to p$ in $Q_T$. Taking the limit as $m\to \infty$ of the sequence of classical solutions to problem \eqref{eq:main-reg} with $\textbf{data}_m$, we obtain a strong solution of problem \eqref{eq:main-reg} corresponding to $\textbf{data}$.

\begin{proposition}[\cite{Ar-Sh-2024}, Theorems~6.1-6.2]
\label{pro:exist-reg} Let conditions \eqref{eq:data-0}-\eqref{eq:data} be fulfilled. For every $f\in L^{2}(Q_T)$, $u_0\in
W^{1,p_0(\cdot)}_0(\Omega)$, and $\epsilon\in (0,1)$, problem
\eqref{eq:main-reg} has a strong solution

\[
u_\epsilon(z)=\lim_{m\to \infty}u_{\epsilon m}(z),\qquad \text{$\nabla u_{\epsilon m}\to \nabla u$ in $L^{p(\cdot)}(Q_T)$}.
\]
The solution $u_\epsilon$ satisfies the estimates

\begin{equation}
\label{eq:unif-eps}
\begin{split}
\operatorname{ess}\sup_{(0,T)}\|u_\epsilon\|_{2,\Omega} & + \operatorname{ess}\sup_{(0,T)}\|\nabla u_\epsilon\|_{p(\cdot),\Omega}+\|u_{\epsilon t}\|_{2,Q_T} +\int_{Q_T}(\epsilon^2+|\nabla u_\epsilon|^2)^{p-2}|
(u_{\epsilon})_{ xx}|^2\,dz
\leq C,
\end{split}
\end{equation}

\begin{equation}
\label{eq:high-int}
\int_{Q_T}(\epsilon^2+|\nabla u_\epsilon|^2)^{p(z)-2+\frac{\delta}{2}}|\nabla u_\epsilon|^2\,dz\leq C',
\end{equation}
with any $\delta \in \left(0,\frac{4}{N+2}\right)$. Moreover, for every $i=\overline{1,N}$,
\begin{equation}
\label{eq:high-diff-0}
\begin{split}
& (\epsilon^2+|\nabla u_\epsilon|^2)^{\frac{p(z)-2}{2}}D_i u_\epsilon\in L^2(0,T;W^{1,2}(\Omega)),
\\
&
 \left\|(\epsilon^2+|\nabla u_\epsilon|^2)^{\frac{p(z)-2}{2}}D_i
u_\epsilon\right\|_{L^2(0,T;W^{1,2}(\Omega))}\leq C''.
\end{split}
\end{equation}
The constants $C$, $C'$, $C''$ depend on $N$, $\partial \Omega$, $L_p$, the structural constants in \eqref{eq:data-0}-\eqref{eq:data}, $\|f\|_{2,Q_T}$, $\|\nabla u_0\|_{p(\cdot,0),\Omega}$, and are independent of $\epsilon$; the constant $C'$ depends also on $\delta$.
\end{proposition}

The solution $u$ of the limit problem \eqref{eq:main} is obtained as the limit $\epsilon_k\to 0$ of a sequence $\{u_{\epsilon_k}\}$ of strong solutions to the regularized problem \eqref{eq:main-reg} corresponding to $\textbf{data}$.

\begin{proposition}[Sec.6, \cite{Ar-Sh-2024}]
\label{pro:existence-degenerate} Assume that conditions
\eqref{eq:data} are fulfilled and $\partial\Omega\in
C^{2+\alpha}$, $\alpha\in (0,1)$. Denote by $\{u_\epsilon\}$ the family of strong solutions of problem \eqref{eq:main-reg} with $f \in L^2(Q_T)$ and
$u_{0}\in W_0^{1,p_0(\cdot)}(\Omega)$. Then problem \eqref{eq:main} has a strong
solution

\[
u =\lim_{k\to \infty}u_{\epsilon_k},\qquad \text{$\nabla u_{\epsilon_k}\to \nabla u$ in $L^{p(\cdot)}(Q_T)$}.
\]
The solution $u$ satisfies the estimate

\begin{equation}
\label{eq:unif-est-degenerate}
\begin{split}
\operatorname{ess}\sup_{(0,T)}\|u\|_{2,\Omega} & +\operatorname{ess}\sup_{(0,T)}\|\nabla u\|_{p(\cdot),\Omega}+\|u_{t}\|_{2,Q_T}
\leq C
\end{split}
\end{equation}
with a constant $C$ depending only on  $N$, $\partial\Omega$, the structural
constants in conditions \eqref{eq:data}, $\|f\|_{2,Q_T}$, and  $\|\nabla
u_{0}\|_{p(\cdot,0),\Omega}$. The solution
possesses the property of  global higher integrability of the gradient:

\begin{equation}
\label{eq:higher-m} \int_{Q_T}|\nabla u|^{2(p(z)-1)+\delta}\,dz\leq C\qquad \text{with
any $\delta\in \left(0,\frac{4}{N+2}\right)$}
\end{equation}
and a constant $C$ depending on $\delta$ and the same quantities as the constant in \eqref{eq:unif-est-degenerate}.
Moreover,

\begin{equation}
\label{eq:imp-reg-previous}
|\nabla u|^{p(z)-2}\nabla u\in L^2(0,T;W^{1,2}(\Omega)) \quad \text{and} \quad \||\nabla u|^{p(z)-2}\nabla u\|_{L^2(0,T;W^{1,2}(\Omega))}\leq C
\end{equation}
with a constant $C$ depending on $N$, the constants in conditions \eqref{eq:data}, $\|\nabla u_{0}\|_{p_0(\cdot),\Omega}$ and  $\|f\|_{2,Q_T}$.
\end{proposition}

The procedure of obtaining a solution to problem \eqref{eq:main} as the limit of a family of classical solutions of the regularized problem \eqref{eq:main-reg} with the data $\textbf{data}_m$ allows one to reduce the study of the regularity properties of this solution to deriving uniform estimates on the classical solutions of the regularized problem \eqref{eq:main-reg} with smooth $\textbf{data}_m$.

\subsection{Main results}
\label{sec:results}

In the proof of the integrability of the gradient we distinguish between the cases \eqref{eq:intr-0} (a) and (b).

\begin{theorem}[Gradient integrability estimates - I]
\label{th:main-1}
Let conditions \eqref{eq:data-0}-\eqref{eq:data} be fulfilled and, in addition,
\[
q^+\leq 1+\frac{p^-}{N}.
\]
If

\[
\text{$f\in L^{N+2}(Q_T)$ \quad and \quad $|\nabla u_0|\in L^r(\Omega)$ with $r\geq \max\{2,p^+\}$},
\]
then problem \eqref{eq:main} has a strong solution $u=\lim\limits_{\epsilon\to 0^+}u_\epsilon$ that satisfies the estimates

\begin{equation}
\label{eq:main-1}
\operatorname{ess}\sup_{(0,T)}\|\nabla u\|_{r,\Omega}\leq C, \qquad
\int_{Q_T} |\nabla u|^{p(z)+r+\rho-2} \,dz \leq C_\rho \quad \text{for any $\rho \in \left(0, \frac{4}{N+2}\right)$}
\end{equation}
 with constants $C$, $C_\rho$ depending only on $\textbf{data}$, $r$, and $C_\rho$ depending also on $\rho$.
\end{theorem}

Theorem \ref{th:main-1} shows that when $f\in L^{N+2}(Q_T)$ the inclusion $|\nabla u(\cdot,t)|\in L^r(\Omega)$ is preserved in time whatever high the order $r$ is. As distinguished from this case, for $f\in L^\sigma(Q_T)$ with $\sigma \in (2,N+2)$ the inclusion $|\nabla u(\cdot,t)|\in L^r(\Omega)$ is preserved in time only for $r$ from a bounded interval whose limits depend on $\sigma$, $N$, and the properties of $p(z)$.

\begin{theorem}[Gradient integrability estimates - II]
\label{th:main-1-1}
Let conditions \eqref{eq:data-0}-\eqref{eq:data} be fulfilled and

\[
\text{$f\in L^{\sigma}(Q_T)$ with $\sigma \in (2,N+2)$ and $|\nabla u_0|\in L^r(\Omega)$}
\]
with $r$ satisfying the inequalities

\begin{equation}
\label{eq:filter}
R_p:=\max\{p^+,2\} \leq r  < R_\sigma:= \frac{ N(p^-(\sigma-1) -\sigma + 2)}{N +2 -\sigma}.
\end{equation}
If
\[
q^+\leq 1+p^-\frac{N+2}{\sigma N},
\]
then problem \eqref{eq:main} has a strong solution $u=\lim\limits_{\epsilon\to 0^+}u_\epsilon$ for which

\begin{equation}
\label{eq:main-1-1}
\operatorname{ess}\sup_{(0,T)}\|\nabla u\|_{r,\Omega}\leq C, \qquad \int_{Q_T} |\nabla u|^{p(z)+r+\rho-2} \,dz \leq C_\rho \quad \text{for any} \ \rho \in \left(0, \frac{4}{N+2}\right)
\end{equation}
with constants $C$, $C_\rho$ depending on $\textbf{data}$ and $r$, and $C_\rho$ depending also on $\rho$.
\end{theorem}

\begin{theorem}[Second-order regularity]
\label{th:main-2}
Let the conditions of Theorem \ref{th:main-1} or Theorem \ref{th:main-1-1} be fulfilled. Then

\begin{equation}
\label{eq:main-2}
\begin{split}
& {\rm (i)} \qquad
|\nabla u|^{\frac{p(z)+r-2}{2}}\in L^2(0,T;W^{1,2}(\Omega)),\qquad \left\||\nabla u|^{\frac{p(z)+r-2}{2}}\right\|_{L^2(0,T;W^{1,2}(\Omega))}\leq C,
\\
& {\rm (ii)} \qquad |\nabla u|^{\frac{p(z)+r}{2}-2}\nabla u\in L^2(0,T;W^{1,2}(\Omega))^N,\qquad \left\||\nabla u|^{\frac{p(z)+r}{2}-2}\nabla u\right\|_{L^2(0,T;W^{1,2}(\Omega))^N}\leq C
\end{split}
\end{equation}
with a constant $C$ depending only on $\textbf{data}$ and $r$.
\end{theorem}

\begin{remark}
\label{rem:admiss}
Inequalities \eqref{eq:filter} furnish the interval of the permissible values of $r$ in case of low integrability of the free term $f$. %According to \eqref{eq:filter},
\begin{itemize}
    \item[{\rm (i)}] When $\sigma\nearrow N+2$, the supremum of the admissible values $R_\sigma$ tends to $+\infty$  independently on the properties of $p(z)$.
    \item[{\rm (ii)}]  For $\sigma \in (2,N+2)$, inequality \eqref{eq:filter} entails conditions on the exponent $p(z)$:
\[
\begin{split}
{\rm (a)} & \qquad \text{if $p^+\leq 2$, then $\displaystyle p^-  >  \frac{\sigma}{\sigma-1}- \frac{2(\sigma-2)}{N(\sigma-1)}$},
\\
{\rm (b)} & \qquad \text{if $p^+>2$, then $\displaystyle p^+  < \frac{ N(p^-(\sigma-1) -\sigma + 2)}{N +2 -\sigma}$}.
\end{split}
\]
Moreover, condition {\rm (b)} can be transformed into the restriction on the oscillation $p^+-p^-$ of $p(z)$ in $Q_T$.

\item[{\rm (iii)}] As $R_\sigma\searrow p^-$ when $\sigma \searrow 2$, the interval $[R_p, R_2)$ of the admissible values of $r$ is empty. This fact directs to considering the case $r \leq R_p$,  $\sigma=2$, which we leave as an open problem for future study. We notice that in the limiting case $\sigma \searrow 2$, the estimates in \eqref{eq:main-1-1} and \eqref{eq:main-2} match the results of \cite[Theorems 2.1 and 2.2]{Ar-Sh-2024} for $r=p(z)$, and \cite{Cianchi-Maz'ya-2019-1} for $r=p$.
\end{itemize}
\end{remark}

\begin{remark}
    In the special case of constant $p > 2$, condition \eqref{eq:filter} becomes independent of $\sigma$ and reads  $p>\dfrac{N}{N+1}$.
\end{remark}

\begin{remark}
\label{rem:reg-eqn}
The assertions of Theorems \ref{th:main-1}, \ref{th:main-1-1}, \ref{th:main-2} follow from uniform estimates for the solutions of the regularized problem \eqref{eq:main-reg} and the limit passages: first as $m\to \infty$ (approximation of the data), and then as $\epsilon\to 0^+$. The limit as $m\to \infty$ with a fixed $\epsilon>0$ is a strong solution $u_\epsilon$ of the nondegenerate regularized problem \eqref{eq:main-reg} corresponding to $\textbf{data}$. For the solution $u_\epsilon$ of the regularized problem with $\textbf{data}$, Theorems \ref{th:main-1} - \ref{th:main-1-1} hold with $|\nabla u|$ replaced by $(\epsilon^2+|\nabla u|^2)^{\frac{1}{2}}$. For $u_\epsilon$, the inclusions

\[
\begin{split}
& D_i\left((\epsilon^2+|\nabla u_{\epsilon}|^2)^{\frac{p(z)+r-2}{4}}\right)\in L^2(Q_T),
\\
& D_i\left((\epsilon^2+|\nabla u_\epsilon|^2)^{\frac{p(z)+r}{4}-1}D_ju_{\epsilon }\right)\in L^{2}(Q_T),\quad i,j=\overline{1,N},
\end{split}
\]
are proven at the first step of the proof of Theorem \ref{th:main-2}.
\end{remark}

\begin{corollary}
\label{cor:global-Holder}
Under the conditions of Theorem \ref{th:main-1} or \ref{th:main-1-1} estimates \eqref{eq:unif-est-degenerate}, and \eqref{eq:main-1} or \eqref{eq:main-1-1} hold. Fix $\rho\in \left(0, \frac{4}{N+2}\right)$ and consider $u(z)$ as an element of the anisotropic Sobolev space
\[
\mathcal{V}_\rho=\left\{u(z)\in L^{2}(Q_T):\,u_t\in L^2(Q_T),\, D_i u\in L^{p^-+r+\rho-2}(Q_T), \;i=\overline{1,N}\right\},\quad r\geq 2.
\]
Assume that $\dfrac{1}{2}+\dfrac{N}{p^-+r+\rho-2}<1$, which is equivalent to $r>2N-p^--\rho+2$. According to \cite[Th.5]{H-Sch-2009},
$\mathcal{V}_\rho$ is continuously embedded into the H\"older space $C^{\beta}(\overline Q_T)$ with
$\beta=1-\dfrac{N}{p^-+r+\rho-2}$.
It follows that if in the conditions of Theorem \ref{th:main-1} or \ref{th:main-1-1}, $r>2N-p^--\rho+2$, then $u\in C^{\beta}(\overline Q_T)$.
\end{corollary}

The main ingredients of the proofs of Theorems \ref{th:main-1}, \ref{th:main-1-1}, \ref{th:main-2} are a  formula of integration by parts in $\Omega$, and an interpolation inequality, which hold for every sufficiently smooth function. The next section is entirely devoted to deriving the formula of integration by parts, pointwise and integral estimates for the elements of the Hessian matrix. We adapt the arguments from \cite[Subsec.3.1.1]{Grisvard-2011} to show that for every constant $r\geq \max\{2,p^+\}$ and every sufficiently smooth function $u$, which vanishes on $\partial\Omega$,

\begin{equation}
\label{eq:by-parts-intr}
\begin{split}
\int_{\Omega}\Delta^{(\epsilon)}_{p(\cdot)}u & \Delta^{(\epsilon)}_{r}u \,dx
\geq C_0\int_\Omega (\epsilon^2+|\nabla u|^2)^{\frac{p(x)+r}{2}-2}\operatorname{trace}\mathcal{H}^2(u)\,dx +\Phi(x,u,\nabla u),
\end{split}
\end{equation}
where $\mathcal{H}(u)$ is the Hessian matrix of $u$, $C_0$ is a constant independent of $u$.  The function $\Phi$ is independent of the components of $\mathcal{H}(u)$ and can be controlled with the help of the following interpolation inequality: for every $\delta>0$ and any $\rho \in \left(0,\frac{4}{N+2}\right)$

\begin{equation}
\label{eq:inter-intr}
\int_{\Omega} (\epsilon^2+|\nabla u|^{2})^{\frac{p(x)+r+\rho-4}{2}}|\nabla u|^2\,dx\leq \delta \int_{\Omega}(\epsilon^2+|\nabla u|^{2})^{\frac{p(x)+r}{2}-2}\operatorname{trace}\mathcal{H}^2(u)\,dx +C
\end{equation}
with a constant $C$ depending only on $\rho$, $N$, $p^\pm$, $r$, $\delta$, $\|u\|_{2,\Omega}$, and the Lipschitz constant of $p(x)$. Multiplying equation \eqref{eq:main-reg} by $\Delta_r^{(\epsilon)}u$ and using \eqref{eq:by-parts-intr}, \eqref{eq:inter-intr} we obtain a differential inequality for the function $\|(\epsilon^2+|\nabla u|^2)\|_{\frac{r}{2},\Omega}^{\frac{r}{2}}$.  The term $(f+F_\epsilon(z,u,\nabla u))\Delta_r^{(\epsilon)}u$ on the right-hand side is estimated by the terms on the left-hand side of the appearing inequality, provided that the exponents $p$, $q$, $s$, and the functions $f$, $u_0$ satisfy the conditions of Theorem \ref{th:main-1} or \ref{th:main-1-1}. The estimates on the integral of $f\Delta ^{(\epsilon)}_ru$ require different conditions on $p$, $\sigma$, $r$ in the cases $\sigma=N+2$ and $\sigma\in (2,N+2)$. The proof of the second-order regularity \eqref{eq:main-2} relies on the uniform estimates of the type \eqref{eq:by-parts-intr} and the pointwise convergence of the gradients of the classical solutions to the regularized problem.

\section{Estimates involving $D_{ij}^2u$}
\label{sec:Hessian}

\subsection{Integration-by-parts formula}
We begin by formulating an integration by parts formula that involves two arbitrary smooth vectors. Assume $\partial\Omega\in C^2$ and denote by $\vec \nu$ the exterior unit normal to $\partial \Omega$. Fix a point $P\in \partial\Omega$, take $C^2$ curves $\{l_1,\ldots,l_{N-1}\}$ orthogonal at every point in a neighborhood of $P$.  Denote by $\{\vec \tau_1,\ldots,\vec \tau_{N-1}\}$ the unit vectors tangent to $l_i$, and by $s_i$ the curve length along $l_i$. %Let $\vec \nu$ be the normal vector to $\partial\Omega$ at $P$.

At every point of $\partial\Omega$ any smooth vector $\vec a$ can be decomposed into the sum of the tangent and normal components:

\[
\begin{split}
& \vec a=\vec a_\tau +a_\nu\vec \nu,
\quad  \vec a_\tau=\sum_{j=1}^{N-1} a_j\vec \tau_j,\quad a_j=(\vec a,\vec \tau_j),\quad a_\nu=(\vec a,\vec \nu).
\end{split}
\]

\begin{proposition}\label{int-by-parts-pro}
    Let $\partial\Omega\in C^2$ and $\vec a$, $\vec b$ be arbitrary vectors with the components
\begin{equation}
\label{eq:vec-reg}
\notag
a_i\in C^0(\overline{\Omega})\cap C^1(\Omega), \qquad b_i\in C^1(\overline{\Omega})\cap C^2(\Omega).
\end{equation}
Then,
\begin{equation}
\label{eq:double-ee}
\begin{split}
\int_\Omega \operatorname{div}\vec a \operatorname{div}\vec b\,dx  &= -\int_{\partial\Omega} \left(\vec a_\tau \nabla _\tau (\vec b\cdot \vec\nu)+ \vec b_\tau \nabla _\tau (\vec a\cdot \vec\nu)\right)\,dS \\
& \qquad - \int_{\partial\Omega}\mathcal{B}(\vec a_\tau;\vec b_\tau)\,dS-\int_{\partial\Omega}a_{\nu}b_{\nu}\operatorname{trace}\mathcal{B} \,dS + \int_\Omega \sum_{i,j=1}^N D_j a_i D_i b_j\,dx,
\end{split}
\end{equation}
where $\mathcal{B}$ is the second fundamental form of the surface $\partial\Omega.$
\end{proposition}

Equality \eqref{eq:double-ee} is crucial for proving the gradient estimates and second-order regularity results. Similar formulas were derived and used in \cite{Ar-Sh-JDE-2023,Ar-Sh-RACSAM-2023,Ar-Sh-2024,Bogelein-2021,Cianchi-Maz'ya-2019-1}. For the sake of completeness of the presentation, in Appendix \ref{sec:intbyparts} we give the detailed proof of Proposition \ref{int-by-parts-pro}. The proof adapts the arguments in \cite[Sec.3.1.1]{Grisvard-2011} to the situation we study in this work.

Let $u$ be a smooth function defined on $\Omega$, $\partial\Omega\in C^2$, and $u=0$ on $\partial \Omega$. Fix $\epsilon\in (0,1)$ and accept the notation

\[
w_\epsilon=\epsilon^2+|\nabla u|^2.
\]
We will apply Proposition \ref{int-by-parts-pro} to the vectors

\begin{equation}
\label{eq:p+r}
\vec a= w_\epsilon ^{\frac{p(z)-2}{2}}\nabla u,\qquad \vec b=  w_\epsilon ^{\frac{r-2}{2}}\nabla u,\quad r=const\geq 2.
\end{equation}
To apply formula \eqref{eq:double-ee} we need $a_i\in C^1(\overline \Omega)\cap C^2(\Omega)$. The straightforward computation gives

\[
\begin{split}
D_ja_i & = w_\epsilon ^{\frac{p-2}{2}}D^2_{ij}u+(p-2) w_\epsilon ^{\frac{p-2}{2}-1}D_iu\sum_{k=1}^ND_kuD^2_{kj}u
+ \frac{1}{2} w_\epsilon ^{\frac{p-2}{2}}\ln w_\epsilon D_iuD_jp.
\end{split}
\]

For every $\epsilon>0$ the inclusion $a_i\in C^1(\overline{\Omega})$ is fulfilled if $u\in C^2(\overline{\Omega})$, $p\in C^1(\overline\Omega)$, while $a_i\in C^2({\Omega})$ holds if $u\in C^3({\Omega})$, $p\in C^2(\Omega)$. Since $u=0$ on $\partial\Omega$, then
\[
\begin{split}
&
a_\nu=\vec{a}\cdot \vec{\nu}=\vec{a}\cdot \dfrac{\nabla u}{|\nabla u|}= w_\epsilon ^{\frac{p(z)-2}{2}}|\nabla u|,\qquad \vec a_\tau=\vec a-a_\nu \vec\nu=0,
\\
& b_\nu=\vec{b}\cdot \vec{\nu}=\vec{b}\cdot \dfrac{\nabla u}{|\nabla u|}= w_\epsilon ^{\frac{r-2}{2}}|\nabla u|,\qquad \vec b_\tau=\vec b-b_\nu \vec\nu=0,
\end{split}
\]
and the integral over $\partial\Omega$ in \eqref{eq:double-ee} transforms into

\[
-\int_{\partial\Omega} w_\epsilon ^{\frac{p+r}{2}-2}|\nabla u|^2\operatorname{trace}\mathcal{B}\,dS.
\]
Formula \eqref{eq:double-ee} with $\vec a$, $\vec b$ defined in \eqref{eq:p+r} takes on the form

\begin{equation}
\label{eq:double-final-e}
\begin{split}
\int_{\Omega} \operatorname{div}\left( w_\epsilon ^{\frac{p-2}{2}}\nabla u\right)\operatorname{div}\left( w_\epsilon ^{\frac{r-2}{2}}\nabla u\right)\,dx
& =- \int_{\partial\Omega} w_\epsilon ^{\frac{p+r}{2}-2}|\nabla u|^2\operatorname{trace}\mathcal{B}\,dS
\\
& +\sum_{i,j=1}^N\int_{\Omega}D_{i}\left( w_\epsilon ^{\frac{p-2}{2}}D_{j}u\right) D_{j}\left( w_\epsilon ^{\frac{r-2}{2}}D_{i}u\right)\,dx.
\end{split}
\end{equation}

\subsection{Pointwise inequalities}
Assume that $u\in C^2(\Omega)$, $u=0$ on $\partial\Omega$, $p\in C^1(\Omega)$, and $p^->1$. By a straightforward computation
\[
\begin{split}
D_{i} \left( w_\epsilon ^{\frac{p-2}{2}}D_{j}u\right)
 & =  w_\epsilon ^{\frac{p-2}{2}}D^2_{ij}u
+ (p-2) w_\epsilon ^{\frac{p-2}{2}-1}D_ju\sum_{k=1}^ND_k u D^2_{ki}u
+\frac{1}{2} w_\epsilon ^{\frac{p-2}{2}}\ln  w_\epsilon  D_juD_ip
\\
& =
 w_\epsilon ^{\frac{p-2}{2}}\left(D^2_{ij}u + (p-2)\eta_j\sum_{k=1}^N D^2_{ki}u\eta_k+ \frac{1}{2}\ln w_\epsilon D_juD_ip\right),
\end{split}
\]

\[
\begin{split}
D_{j} \left( w_\epsilon ^{\frac{r-2}{2}}D_{i}u\right)
&
=  w_\epsilon ^{\frac{r-2}{2}}D^2_{ij}u
+ (r-2) w_\epsilon ^{\frac{r-2}{2}-1}D_iu\sum_{k=1}^ND_k u D^2_{kj}u
=
 w_\epsilon ^{\frac{r-2}{2}}\left(D^2_{ij}u + (r-2)\eta_i\sum_{k=1}^N D^2_{kj}u\eta_k\right)
\end{split}
\]
with

\begin{equation}
\label{eq:u-eta}
\vec{\eta}=\dfrac{\nabla u}{  \sqrt{w_\epsilon} },\qquad |\vec \eta|<1.
\end{equation}
Then

\[
\begin{split}
D_{i} & \left( w_\epsilon ^{\frac{p-2}{2}}D_{j}u\right) D_{j}\left( w_\epsilon ^{\frac{r-2}{2}}D_{i}u\right)
%\\
%&
=  w_\epsilon ^{\frac{p+r}{2}-2} \left(D^2_{ij}u + (p-2)\eta_j\sum_{k=1}^N D^2_{ki}u\eta_k+ \frac{1}{2}  \sqrt{w_\epsilon} \ln (\epsilon^2+ |\nabla u|^2) \eta_jD_ip\right)
\\
& \quad \times
\left(D^2_{ij}u + (r-2)\eta_i\sum_{k=1}^N D^2_{kj}u\eta_k\right)
\equiv  w_\epsilon ^{\frac{p+r}{2}-2}\mathcal{K}_{ij}.
\end{split}
\]
Denote by $\mathcal{H}(u)$ the Hessian matrix of $u$: the symmetric $N\times N$ matrix with the entries $\mathcal{H}_{ij}(u)=D^2_{ij}u$, $i,j=1,\ldots,N$. In the new notation, the expression for $\mathcal{K}_{ij}$ becomes

\begin{equation*}
\begin{split}
\mathcal{K}_{ij} & = \mathcal{H}^2_{ij}(u) + \mathcal{H}_{ij}(u)\left[\sum_{k=1}^N((p-2)\eta_j\mathcal{H}_{ki}(u)+(r-2)\eta_i \mathcal{H}_{kj}(u))\eta_k\right]
\\
& + (p-2)(r-2)\eta_i\eta_j \sum_{k,l=1}^N
\mathcal{H}_{li}(u)\mathcal{H}_{kj}(u)\eta_l\eta_k
\\
& + \frac{1}{2}\mathcal{H}_{ij}(u)\ln w_\epsilon D_iuD_jp
+ \frac{1}{2}(r-2)\eta_i\eta_j  \sqrt{w_\epsilon} \sum_{k=1}^N\mathcal{H}_{ki}(u)\eta_k \ln  w_\epsilon D_jp
\equiv \sum_{s=1}^{5}\mathcal{J}^{(s)}_{ij}.
\end{split}
\end{equation*}
Summing up we obtain

\begin{equation}
\label{eq:K}
\begin{split}
& \sum_{i,j=1}^N \mathcal{K}_{ij} =\sum_{i,j=1}^N\mathcal{H}^2_{ij}(u)
+ (p+r-4)\sum_{i=1}^N\left(\sum_{j=1}^N\mathcal{H}_{ij}(u)\eta_j\right) \left(\sum_{k=1}^N\mathcal{H}_{ik}(u)\eta_{k}\right)
\\
& \qquad + (p-2)(r-2)\sum_{i,j=1}^N\eta_i\eta_j \sum_{k,l=1}^N
\mathcal{H}_{li}(u)\mathcal{H}_{kj}(u)\eta_l\eta_k +\sum_{i,j=1}^N\sum_{s=4}^5\mathcal{J}_{ij}^{(s)}
\\
& = \sum_{i,j=1}^N\mathcal{H}^2_{ij}(u) + (p+r-4)\sum_{i=1}^N\left(\sum_{j=1}^N\mathcal{H}_{ij}(u)\eta_j\right)\left(\sum_{k=1}^N \mathcal{H}_{ik}(u)\eta_k\right)
\\
& \qquad +(p-2)(r-2) \sum_{i,j=1}^N\left(\sum_{l=1}^N
\mathcal{H}_{li}(u)\eta_l\eta_i\right) \left(\sum_{k=1}^{N}\mathcal{H}_{kj}(u)\eta_j\eta_k\right) +\sum_{i,j=1}^N\sum_{s=4}^5\mathcal{J}_{ij}^{(s)}
\\
& = \sum_{i,j=1}^N\mathcal{H}^2_{ij}(u) + (p+r-4)\sum_{i=1}^N\left(\sum_{j=1}^N\mathcal{H}_{ij}(u)\eta_j\right)\left(\sum_{k=1}^N \mathcal{H}_{ik}(u)\eta_k\right)
\\
& \qquad +(p-2)(r-2) \left(\sum_{i,l=1}^N
\mathcal{H}_{li}(u)\eta_l\eta_i\right) \left(\sum_{j,k=1}^{N}\mathcal{H}_{kj}(u)\eta_j\eta_k\right) +\sum_{i,j=1}^N\sum_{s=4}^5\mathcal{J}_{ij}^{(s)}
\\
& = \operatorname{trace}\mathcal{H}^2(u) + (p+r-4)|(\mathcal{H}(u),\vec \eta)|^2+(p-2)(r-2) \left((\mathcal{H}(u),\vec \eta)\cdot \vec \eta\right)^2 +\sum_{i,j=1}^N\sum_{s=4}^5\mathcal{J}_{ij}^{(s)}.
\end{split}
\end{equation}
The residual terms that depend on the derivatives of $p(z)$ have the form

\begin{equation}
\label{eq:residuals}
\begin{split}
& \mathcal{J}_{ij}^{(4)} =\frac{1}{2}\mathcal{H}_{ij}(u)\ln w_\epsilon D_iuD_jp,
\\
& \mathcal{J}_{ij}^{(5)} = \frac{1}{2}(r-2)\sum_{k=1}^N\mathcal{H}_{kj}(u)\eta_k \eta_j\ln  w_\epsilon D_jp D_iu
\end{split}
\end{equation}
and vanish if $p$ is independent of $x$. Accept the notation
\[
\mathcal{F}(\vec \eta)\equiv \operatorname{trace}\mathcal{H}^2(u) + (p+r-4)|(\mathcal{H}(u),\vec \eta)|^2+(p-2)(r-2) \left((\mathcal{H}(u),\vec \eta)\cdot \vec \eta\right)^2,
\]
with $\vec \eta$ is defined in \eqref{eq:u-eta}.
\begin{proposition}
\label{pro:pointwise}
Let $u\in C^2(\Omega)$. If $p\in C^0(\Omega)$, $p^->1$ and $r\geq 2$, then
\begin{equation}
\label{eq:est-from-below}
\mathcal{F}( \vec \eta)\geq \sigma\operatorname{trace}\mathcal{H}^2(u)\equiv \sigma|u_{xx}|^2\qquad \forall\,\vec \eta \in \mathbb{R}^N,\;\;|\vec \eta|\leq 1,
\end{equation}
with the constant $\sigma=\min\{1,p^--1\}$.
\end{proposition}

\begin{proof}
Fix a point $x_0\in \Omega$, denote $p=p(x_0)$, and take an arbitrary vector $\vec \eta$, $|\vec \eta|<1$. If $\mathcal{H}(u(x_0))=0$, inequality \eqref{eq:est-from-below} is obvious. Assume that $\mathcal{H}(u(x_0))\not=0$, diagonalize the matrix $\mathcal{H}(u(x_0))$ by means of rotation,  and denote by $\vec \zeta$ the vector $\vec \eta$ in the new basis:

\[
(\mathcal{H}(u(x_0))\cdot \vec \eta )\cdot (\mathcal{H}(u(x_0))\cdot \vec \eta )=\sum_{i=1}^Nd_i^2\zeta_i^2,\qquad \mathcal{H}(u(x_0))\cdot\vec \eta)\cdot \vec \eta=\sum_{i=1}^N d_i\zeta_i^2,\qquad \operatorname{trace}\mathcal{H}^2(u)=\sum_{i=1}^N d_i^2,
\]
where $d_i$ are the diagonal elements of the transformed matrix. Set $|\vec \zeta|=\lambda<1$ and denote
\[
\vec \xi=\lambda^{-1}\vec \zeta, \quad |\vec \xi|=1.
\]
Then

\[
\begin{split}
\mathcal{F}(\vec \eta) &  = \sum_{i=1}^Nd_i^2+(p+r-4)\sum_{i=1}^N d^2_i\zeta_i^2 + (p-2)(r-2) \left(\sum_{i=1}^N d_i\zeta_i^2\right)^2,
\\
& = \sum_{i=1}^Nd_i^2+\lambda^2(p+r-4)\sum_{i=1}^N d_i^2\xi_i^2 + \lambda^4(p-2)(r-2) \left(\sum_{i=1}^N d_i\xi_i^2\right)^2
\\
& \equiv
\Phi(\vec \xi,\lambda),\qquad |\vec \xi|=1,\quad \lambda \in (0,1).
\end{split}
\]
If $p\geq 2$, $r\geq 2$, the second and third terms of $\Phi(\vec \xi,\lambda)$ are nonnegative, whence

\[
\mathcal{F}(\vec \eta)\geq \sum_{i=1}^N d_i^2=\operatorname{trace} \mathcal{H}^2(u(x_0))
\]
as required. Let

\[
1<p^-\leq p<2,\qquad r\geq 2.
\]
By the Cauchy-Bunyakovsky-Schwarz inequality

\begin{equation}
\label{eq:F-new-1}
\left(\sum_{i=1}^Nd_i\xi_i^2\right)^2 =\left(\sum_{i=1}^N\left(d_i\xi_i\right)\xi_i\right)^2\leq \left(\sum_{i=1}^Nd_i^2\xi_i^2\right)\left(\sum_{i=1}^N\xi_i^2\right)=\sum_{i=1}^Nd_i^2\xi_i^2.
\end{equation}
Since $\lambda <1$ and $|\vec \xi|=1$, the first and third terms of $\Phi$ are bounded from below:

\[
\begin{split}
& \lambda^4(p-2)(r-2) \left(\sum_{i=1}^N d_i\xi_i^2\right)^2\geq \lambda^2 (p^--2)(r-2) \sum_{i=1}^Nd_i^2\xi_i^2,
\\
& \sum_{i=1}^Nd_i^2=  \sigma\sum_{i=1}^Nd_i^2+(1-\sigma)\sum_{i=1}^Nd_i^2\geq \sigma\sum_{i=1}^Nd_i^2+(1-\sigma)\lambda^2\sum_{i=1}^Nd_i^2\xi_i^2
\end{split}
\]
with any $\sigma\in [0,1]$. Then

\[
\Phi(\vec \xi,\lambda)\geq \lambda^2\left((1-\sigma) +(p^-+r-4)+(r-2)(p^--2)\right)\sum_{i=1}^Nd_i^2\xi_i^2 +\sigma \sum_{i=1}^Nd_i^2.
\]
For $\sigma= \min\{1,p^--1\}$

\[
(1-\sigma) +(p^-+r-4)+(r-2)(p^--2)=-\sigma+(r-1)(p^--1)\geq -\sigma+(p^--1)\geq 0,
\]
whence $\displaystyle \Phi(\vec \xi,\lambda)\geq \sigma \sum_{i=1}^Nd_i^2$.
\end{proof}

\begin{lemma}
\label{le:pointwise-1}
If $u\in C^2(\Omega)$, $p\in C^1(\Omega)$, $p^->1$, and $r\geq 2$, then

\begin{equation}
\label{eq:est-from-below-1}
\begin{split}
\sum_{i,j=1}^N & D_{i} \left( w_\epsilon ^{\frac{p-2}{2}}D_{j}u\right) D_{j}\left( w_\epsilon ^{\frac{r-2}{2}}D_{i}u\right)
%\\
%&
\geq \sigma w_\epsilon ^{\frac{p+r}{2}-2}|u_{xx}|^2
+ w_\epsilon ^{\frac{p+r}{2}-2}\sum_{i,j=1}^N\sum_{s=4}^5\mathcal{J}_{ij}^{(s)}
\end{split}
\end{equation}
with $\sigma=\min\{1,p^--1\}$ and $\mathcal{J}_{ij}^{(s)}$ defined in \eqref{eq:residuals}.
\end{lemma}

\begin{proof}
It is sufficient to combine formulas \eqref{eq:K} with inequality \eqref{eq:est-from-below}.
\end{proof}
\subsection{Integral inequalities}
Plugging \eqref{eq:est-from-below-1} into \eqref{eq:double-final-e} we arrive at the following assertion.

\begin{lemma}
\label{le:principal-e} Let $\partial\Omega\in C^2$, $u\in C^3(\Omega)\cap
C^{2}(\overline{\Omega})$ and $u=0$ on $\partial\Omega$. If $p\in C^1(\overline{\Omega})$, $p^-> 1$,  and $r\geq 2$, then

\begin{equation}
\label{eq:p-est-1}
\begin{split}
\int_{\Omega} \operatorname{div}\left(w_\epsilon^{\frac{p-2}{2}}\nabla u\right) &
\operatorname{div}\left(w_\epsilon^{\frac{r-2}{2}}\nabla
u\right)\,dx
 \geq - \int_{\partial\Omega}w_\epsilon^{\frac{p+r}{2}-2}|\nabla u|^2\operatorname{trace}\mathcal{B}\,dS
 \\
 &
  +
\sigma \int_{\Omega}w_\epsilon^{\frac{p+r}{2}-2}|u_{xx}|^2\,dx
 + \sum_{i,j=1}^N\sum_{s=4}^5 \int_{\Omega}\mathcal{J}_{ij}^{(s)}\,dx,
\end{split}
\end{equation}
where $\mathcal{B}$ is the second fundamental form of the surface $\partial\Omega$, $\sigma=\min\{1,p^--1\}$, and $\mathcal{J}_{ij}^{(s)}$ are defined in \eqref{eq:residuals}.
\end{lemma}

Observe that

\begin{equation}
\label{eq:F-1}
\begin{split}
w_\epsilon|u_{xx}|^2 & \geq |\nabla u|^2|u_{xx}|^2 =\left(\sum_{k=1}^N\left(D_ku\right)^2\right)\left(\sum_{i,j=1}^N\left(D^2_{ij}u\right)^2\right)
\geq \sum_{i=1}^N\left(\sum_{j=1}^N D^2_{ij}uD_ju\right)^2
\\
& =\frac{1}{4}\sum_{i=1}^N\left(D_i\left(\epsilon^2+\sum_{j=1}^N \left(D_j u\right)^2\right)\right)^2 \equiv \frac{1}{4}|\nabla w_\epsilon|^2
\end{split}
\end{equation}
and

\[
\begin{split}
\left|\nabla \left(w_\epsilon^{\frac{p+r-2}{4}}\right)\right|^2 & = \frac{(p+r-2)^2}{16}w_\epsilon^{\frac{p+r}{2}-3}|\nabla w_\epsilon|^2
\\
&
 \quad
+ \frac{p+r-2}{8}w_\epsilon^{\frac{p+r}{2}-2}\ln w_\epsilon (\nabla w_\epsilon,\nabla p)
 +\frac{1}{16}w_\epsilon^{\frac{p+r}{2}-1}\ln^2w_\epsilon|\nabla p|^2.
\end{split}
\]
Under the assumptions of Lemma \ref{le:principal-e} the second term on the right-hand side of \eqref{eq:p-est-1} is bounded from below:

\begin{equation}
\label{eq:F-2}
\begin{split}
w_\epsilon^{\frac{p+r}{2}-2}|u_{xx}|^2
& \geq
 \dfrac{1}{4} w_\epsilon^{\frac{p+r}{2}-3}|\nabla w_\epsilon|^2
 \geq \frac{4}{(p+r-2)^2} \left|\nabla \left(w_\epsilon^{\frac{p+r-2}{4}}\right)\right|^2
 \\
 & \quad
- \frac{2L_p}{p+r-2} w_\epsilon^{\frac{p+r}{2}-2}|\ln w_\epsilon||\nabla w_\epsilon|
- \frac{4L_p^2}{(p+r-2)^2}  w_\epsilon^{\frac{p+r}{2}-1}\ln^2 w_\epsilon
\\
& \equiv \frac{4}{(p+r-2)^2} \left|\nabla \left(w_\epsilon^{\frac{p+r-2}{4}}\right)\right|^2 + w_\epsilon^{\frac{p+r}{2}-2}\left(\mathcal{M}_1 +\mathcal{M}_2\right)
\end{split}
\end{equation}
with

\begin{equation}
\label{eq:residuals-1}
\mathcal{M}_1=\frac{2L_p}{p+r-2} |\ln w_\epsilon||\nabla w_\epsilon|,\qquad \mathcal{M}_2=\frac{4L_p^2}{(p+r-2)^2}  w_\epsilon\ln^2 w_\epsilon.
\end{equation}
Inequality \eqref{eq:p-est-1} becomes

\begin{equation}
\label{eq:p-est-1-trans}
\begin{split}
\int_{\Omega} \operatorname{div}\left(w_\epsilon^{\frac{p-2}{2}}\nabla u\right) &
\operatorname{div}\left(w_\epsilon^{\frac{r-2}{2}}\nabla
u\right)\,dx
 \geq - \int_{\partial\Omega}w_\epsilon^{\frac{p+r}{2}-2}|\nabla u|^2\operatorname{trace}\mathcal{B}\,dS
 \\
 &
  +
\frac{4\sigma}{(p^++r-2)^2} \int_{\Omega}\left|\nabla \left(w_\epsilon^{\frac{p+r-2}{4}}\right)\right|^2\,dx
\\
&
+ \sigma \int_\Omega w_\epsilon^{\frac{p+r}{2}-2}\left(\mathcal{M}_1 +\mathcal{M}_2\right)\,dx
 + \sum_{i,j=1}^N\sum_{s=4}^5 \int_{\Omega}\mathcal{J}_{ij}^{(s)}\,dx.
\end{split}
\end{equation}

\subsection{Interpolation inequalities}
To control the right-hand side of \eqref{eq:p-est-1-trans} we will use the following inequality.

 \begin{lemma}[Lemma 4.5, \cite{Ar-Sh-RACSAM-2023}]
\label{le:racsam}
Let $\partial\Omega\in C^2$, $q\in C^{0,1}(\Omega)$ with the Lipschitz constant $L_q$,  and $q^->\frac{2N}{N+2}$. For every $u\in C^2(\Omega)\cap C^1(\overline{\Omega})$, $\delta\in (0,1)$ and $\rho \in \left(0,\frac{4}{N+2}\right)$

\[
\int_{\Omega} w_\epsilon ^{\frac{q(x)+\rho-2}{2}}|\nabla u|^2\,dx\leq \delta \int_{\Omega}(\epsilon^2+|\nabla u|^{2})^{\frac{q(x)-2}{2}}|u_{xx}|^2\,dx +C
\]
with a constant $C$ depending only on $\rho$, $N$, $q^\pm$, $L_q$, $\delta$, $\|u\|_{2,\Omega}$.
\end{lemma}

To apply Lemma \ref{le:racsam} in \eqref{eq:p-est-1-trans} with $r\geq 2$ we set

\begin{equation}
\label{eq:lower-bound-p}
\notag
\frac{q(z)-2}{2}=\frac{p(z)+r-4}{2}\qquad \text{and claim}\quad q(z)>\frac{2N}{N+2}\quad \Leftrightarrow \quad p(z)+r>2+\frac{2N}{N+2}.
\end{equation}
The last inequality is surely true because $p^->\frac{2(N+1)}{N+2}$ by assumption and $r\geq 2$. Then for every $\delta>0$

\begin{equation}
\label{eq:interpol-1}
\int_{\Omega} w_\epsilon^{\frac{p+r+\rho-4}{2}}|\nabla u|^2\,dx\leq \delta \int_{\Omega} w_\epsilon^{\frac{p+r}{2}-2}|u_{xx}|^2\,dx +C,\quad \rho \in \left(0,\frac{4}{N+2}\right).
\end{equation}
If $\rho \in \left(\frac{2}{N+2},\frac{4}{N+2}\right)$, inequality \eqref{eq:interpol-1} remains true with the integrand of the left-hand side substituted by $|\nabla u|^{p+r+\rho-2}$ and, by Young' inequality, extends to the range $\rho \in \left(\left.0,\frac{2}{N+2}\right]\right.$:

\begin{equation}
\label{eq:interpol-2}
\int_{\Omega} |\nabla u|^{p+r+\rho-2}\,dx \leq \delta \int_{\Omega} w_\epsilon^{\frac{p+r}{2}-2}|u_{xx}|^2\,dx +C,\quad \rho \in \left(0,\frac{4}{N+2}\right),
\end{equation}
with $C=C(\rho,N,p^\pm,r,L,\delta,\|u\|_{2,\Omega})$.

We will repeatedly use the following elementary inequality: for every $\gamma>0$ and $\alpha\in (0,\gamma)$ there is a constant $C=C(\alpha,\gamma)$ such that

\begin{equation}
\label{eq:elem}
s^\gamma |\ln s|\leq C(1+s^{\gamma+\alpha})\quad \text{for $s>0$}.
\end{equation}
It follows from \eqref{eq:elem} and \eqref{eq:interpol-1} that for every $\lambda \in \left(0,\frac{2}{N+2}\right)$ and $\delta\in (0,1)$

\begin{equation}
\label{eq:interpol-2-prim}
\begin{split}
\int_{\Omega} w_\epsilon^{\frac{p+r-4}{2}}|\nabla u|^2 |\ln w_\epsilon|\,dx & \leq C\left(1+\int_{\Omega} w_\epsilon^{\frac{p+r-4}{2}+\lambda}|\nabla u|^2\,dx\right)
%\\
%&
\leq \delta \int_{\Omega} w_\epsilon^{\frac{p+r}{2}-2}|u_{xx}|^2\,dx +C
\end{split}
\end{equation}
with a constant $C$ depending on the same quantities as the constants in \eqref{eq:interpol-1}, \eqref{eq:interpol-2}, and on $\lambda$.

Inequality \eqref{eq:interpol-1} remains valid for $s=0$ because for every $\gamma\in \left(0,\frac{4}{N+2}\right)$

\begin{equation}
  \label{eq:w-u}
\begin{split}
   w_\epsilon^{\frac{p+r}{2}-2}|\nabla u|^2 & \leq 1+ w_\epsilon^{\frac{p+r+\gamma}{2}-1}
    \leq 1+\begin{cases}
     (2\epsilon^2)^{\frac{p+r+\gamma}{2}-1} & \text{if $|\nabla u|<\epsilon$},
     \\
     w_\epsilon^{\frac{p+r+\gamma}{2}-2}(2|\nabla u|^2) & \text{if $|\nabla u|\geq \epsilon$}
     \end{cases}
     \\
     & \leq C\left(1+ w_\epsilon^{\frac{p+r+\gamma}{2}-2}|\nabla u|^2\right).
\end{split}
\end{equation}

\subsection{Estimating $\mathcal{J}_{ij}^{(4,5)}$ and $\mathcal{M}_{1,2}$} In \eqref{eq:p-est-1-trans}, the residual terms $\mathcal{J}_{ij}^{(4,5)}$  are defined by formulas \eqref{eq:residuals}, the terms $\mathcal{M}_{1,2}$ are defined in \eqref{eq:residuals-1}. Since $|\eta|<1$, $\mathcal{J}_{i,j}^{(4,5)}$ satisfy the inequalities

\[
\begin{split}
& |\mathcal{J}_{ij}^{(4)}| \leq L_p|\ln w_\epsilon| |\nabla u||u_{xx}|,
\\
& |\mathcal{J}_{ij}^{(5)}| \leq  L_pN(r-2)|\ln w_\epsilon||\nabla u||u_{xx}|,\qquad i,j=\overline{1,N},
\end{split}
\]
where $L_p$ is the Lipschitz constant of $p(z)$. The terms $\mathcal{M}_{1,2}$ are bounded by

\[
\begin{split}
|\mathcal{M}_1| & \leq C |\ln w_\epsilon| |\nabla u||u_{xx}|,
\qquad
|\mathcal{M}_2|
\leq C' w_\epsilon \ln^2 w_\epsilon
\end{split}
\]
with constants $C$, $C'$ depending on $N$, $p^\pm$, $r$. It follows from \eqref{eq:interpol-1}, \eqref{eq:interpol-2}, and the Young inequality that for every $\delta>0$

\begin{equation}
\label{eq:est-J-4-5}
\begin{split}
\int_\Omega & w_\epsilon^{\frac{p+r}{2}-2}\left(|\mathcal{J}^{(4)}_{ij}| +|\mathcal{J}^{(5)}_{ij}|\right)\,dx
 \leq C (r+1) \int_\Omega w_\epsilon^{\frac{p+r-4}{2}}|\nabla u||u_{xx}||\ln w_\epsilon|\,dx
\\
& = C (r+1) \int_{\Omega}\left(w_\epsilon^{\frac{p+r-4}{4}}|u_{xx}|\right) \left(w_\epsilon^{\frac{p+r-4}{4}}|\nabla u||\ln w_\epsilon|\right)\,dx
\\
&
\leq \delta \int_\Omega w_\epsilon^{\frac{p+r-4}{2}}|u_{xx}|^2+C (r+1)^2 \int_\Omega w_\epsilon^{\frac{p+r+2\lambda-4}{2}}|\nabla u|^2\,dx+C'
%\\
%& \leq 2\delta \int_\Omega w_\epsilon ^{\frac{p+r-4}{2}}|u_{xx}|^2 +C''(r,\delta).
\end{split}
\end{equation}
with independent of $r$ constants $C$, $C'=C'(r,\delta,\lambda)$. By virtue of \eqref{eq:interpol-1}, the last integral is bounded for every $0<\lambda<\dfrac{2}{N+2}$:

\[
(r+1)^2 \int_\Omega w_\epsilon^{\frac{p+r+2\lambda-4}{2}}|\nabla u|^2\,dx \leq \delta'\int_\Omega w_\epsilon^{\frac{p+r}{2}-2}|u_{xx}|^2\,dx + C''
\]
with a constant $C''$ depending on $r$. The estimate for the integral of  $w_\epsilon^{\frac{p+r}{2}-2}\mathcal{M}_{1}$ follows from \eqref{eq:est-J-4-5}. To estimate $w_\epsilon^{\frac{p+r}{2}-2}\mathcal{M}_{2}$ we apply \eqref{eq:interpol-2-prim}.

\subsection{The boundary integrals}
The boundary integral in \eqref{eq:residuals-1} needs special estimating only if $\Omega$ is not convex. Let $\partial\Omega\in C^2$ but not necessarily convex. Since $|\operatorname{trace}\mathcal{B}|\leq K$ with a finite constant $K>0$, estimating the boundary integral in \eqref{eq:p-est-1} amounts to estimating the integral of $ w_\epsilon ^{\frac{p+r}{2}-2}|\nabla u|^2$ over $\partial\Omega$.  By \cite[Lemma 1.5.1.9]{Grisvard-2011} there exist a constant $\gamma>0$ and a function $\vec \mu\in C^{\infty}(\overline{\Omega})^N$ such that $\vec \mu\cdot\nu \geq \gamma>0$ on $\partial\Omega$. Then

\[
\begin{split}
\gamma \int_{\partial\Omega} & w_\epsilon^{\frac{p+r}{2}-2}|\nabla u|^2\,dS  \leq \int_{\Omega}\operatorname{div}\left(w_\epsilon^{\frac{p+r}{2}-2}|\nabla u|^2\vec \mu\right)\,dx
 \\
 &
 = \int_{\Omega} \vec \mu \cdot\nabla \left(w_\epsilon^{\frac{p+r}{2}-2}|\nabla u|^2\right)\,dx + \int_{\Omega}(\operatorname{div}\vec \mu )w_\epsilon^{\frac{p+r}{2}-2}|\nabla u|^2\,dx
\\
& \leq C_1\int_{\Omega}w_{\epsilon}^{\frac{p+r}{2}-2}|\nabla u|^2\,dx
+ C_2\int_{\Omega} \left(\frac{|p+r-4|}{2}w_\epsilon^{\frac{p+r}{2}-3}|\nabla u|^2+w_\epsilon^{\frac{p+r}{2}-2}\right)|\nabla u||u_{xx}| \,dx
\\
&
\quad + C_3\int_{\Omega} w_\epsilon^{\frac{p+r}{2}-2}|\nabla u|^2 |\ln w_\epsilon||\nabla p|\,dx
\\
& \equiv C_1\mathcal{I}_1 +C_2\mathcal{I}_2 + C_3\mathcal{I}_3.
\end{split}
\]
To estimate $\mathcal{I}_2$ we apply the Cauchy inequality:

\[
\begin{split}
\mathcal{I}_2 & \leq (p^++r-1)\int_\Omega w_\epsilon^{\frac{p+r}{2}-2}|\nabla u||u_{xx}|\,dx
\\
&
=(p^++r-1)\int_{\Omega} \left(w_\epsilon ^{\frac{p+r}{4}-1}|u_{xx}|\right) \left(w_\epsilon^{\frac{p+r}{4}-1}|\nabla u|\right)\,dx
\\
& \leq \sigma  \int_{\Omega} w_\epsilon^{\frac{p+r}{2}-2}|u_{xx}|^2\,dx + C_\sigma (p_++r-1)^2\int_\Omega w_\epsilon^{\frac{p+r}{2}-2}|\nabla u|^2\,dx
\\
& \leq
\sigma  \int_{\Omega} w_\epsilon^{\frac{p+r}{2}-2}|u_{xx}|^2\,dx + C' (p^++r)^2 \mathcal{I}_1
\end{split}
\]
with an arbitrary $\sigma>0$ and an independent of $r$ constant $C'$. To estimate $\mathcal{I}_1$ we apply \eqref{eq:interpol-1} with $s=0$ (see \eqref{eq:w-u}). The integral $\mathcal{I}_3$ is estimated by virtue of \eqref{eq:elem} and \eqref{eq:interpol-1}. Gathering these inequalities we estimate the boundary integral as follows:

\begin{equation}
\label{eq:boundary-integral}
\begin{split}
\gamma \int_{\partial\Omega} w_\epsilon^{\frac{p+r}{2}-2}|\nabla u|^2\,dS & \leq \alpha \int_{\Omega} w_\epsilon^{\frac{p+r}{2}-2}|u_{xx}|^2\,dx
+C
\end{split}
\end{equation}
with any $\alpha>0$ and an independent of $\epsilon$ constant $C=C\left(\alpha, N, r, p^+, L_p, \|u\|_2,\Omega\right)$. We may now refine Lemma \ref{le:principal-e}.
\begin{lemma}
\label{le:principal-improved}
Let in the conditions of Lemma \ref{le:principal-e} $p^->\dfrac{2N}{N+2}$ and $r\geq 2$. There exist constants $C_1,C_1'>0$ and a positive constant $C_2=C_2(N,r,p^\pm,L,\partial\Omega,\|u\|_{2,\Omega})$ such that the function $w_\epsilon\equiv \epsilon^2+|\nabla u|^2$ satisfies the inequalities

\begin{equation}
\label{eq:p-est-1-prim}
\begin{split}
C_1\int_{\Omega} w_\epsilon^{\frac{p+r}{2}-2} |u_{xx}|^2\,dx
\leq \int_{\Omega}\operatorname{div}\left(w_\epsilon^{\frac{p-2}{2}}\nabla u\right)\operatorname{div}\left(w_\epsilon^{\frac{r-2}{2}}\nabla u\right)\,dx +C_2.
\end{split}
\end{equation}

\begin{equation}
\label{eq:p-est-2-prim}
\begin{split}
C'_1\int_{\Omega} \left|\nabla \left(w_\epsilon^{\frac{p+r-2}{4}}\right)\right|^2\,dx
\leq \int_{\Omega}\operatorname{div}\left(w_\epsilon^{\frac{p-2}{2}}\nabla u\right)\operatorname{div}\left(w_\epsilon^{\frac{r-2}{2}}\nabla u\right)\,dx +C_2.
\end{split}
\end{equation}
\end{lemma}

\begin{proof}
Inequality \eqref{eq:p-est-1-prim} follows from \eqref{eq:p-est-1} by using estimates \eqref{eq:est-J-4-5}, the estimates on $\mathcal{M}_{1,2}$, and estimate \eqref{eq:boundary-integral}.
\end{proof}

\begin{corollary}
\label{cor:p-prim} Let the functions $u$, $p$ and the constant $r$ satisfy the conditions of Lemma \ref{le:principal-improved} for a.e. $t\in (0,T)$. Then

\begin{equation}
\label{eq:est-cyl}
\begin{split}
& c_1\int_{Q_T} w_\epsilon^{\frac{p(z)+r}{2}-2} |u_{xx}|^2\,dz
%\\
%&
\leq \int_{Q_T}\operatorname{div}\left(w_\epsilon^{\frac{p(z)-2}{2}}\nabla u\right)\operatorname{div}\left(w_\epsilon^{\frac{r-2}{2}}\nabla u\right)\,dz +c_2,
\end{split}
\end{equation}

\begin{equation}
\label{eq:est-cyl-1}
\begin{split}
c'_1\int_{Q_T} \left|\nabla \left(w_\epsilon^{\frac{p(z)+r-2}{4}}\right)\right|^2\,dz
%\\
%&
\leq \int_{Q_T}\operatorname{div}\left(w_\epsilon^{\frac{p(z)-2}{2}}\nabla u\right)\operatorname{div}\left(w_\epsilon^{\frac{r-2}{2}}\nabla u\right)\,dz +c'_2
\end{split}
\end{equation}
with the constants $c_1,c_1',c_2, c_2'$ depending on the same quantities as the constants in \eqref{eq:p-est-1-prim}, $\operatorname{ess}\sup\limits_{(0,T)}\|u\|_{2,\Omega}^2$ and $T$.
\end{corollary}

\section{Uniform estimates on the gradient of a classical solution}\label{sec:est-source}
In this section, we establish the global uniform estimates on the gradients of classical solutions to the regularized problem \eqref{eq:main-reg} with smooth data $\textbf{data}_m$ and $\epsilon\in (0,1)$. For the convenience of presentation, we omit the subindex $m$ of the exponents $p$, $q$, $s$.

\subsection{Equation without nonlinear sources - I: the case $\sigma= N+2$} Let $u$ be the classical solution of problem \eqref{eq:main-reg} with smooth $\textbf{data}_m$. We begin with the model case $F_\epsilon=0$. Set $w_\epsilon= \epsilon^2+|\nabla u|^2$, multiply equation \eqref{eq:main-reg} by $-\div\left(w_\epsilon^{\frac{r-2}{2}} \nabla u\right)$ with $r \geq  \max\{2,p^+\}$, and integrate the result over $\Omega$:

\[
\begin{split}
\frac{1}{r}\dfrac{d}{dt} & \int_{\Omega}w_\epsilon^{\frac{r}{2}}\,dx + \int_{\Omega}\operatorname{div}\left(w_\epsilon^{\frac{p-2}{2}}\nabla u\right)\operatorname{div}\left(w_\epsilon^{\frac{r-2}{2}}\nabla u\right)\,dx
= -\int_{\Omega}f \operatorname{div}\left(w_\epsilon^{\frac{r-2}{2}}\nabla u\right)\,dx.
\end{split}
\]
By \eqref{eq:p-est-1} and \eqref{eq:p-est-1-prim}, \eqref{eq:p-est-2-prim}

\begin{equation}
\label{eq:a-priori-1}
\begin{split}
\frac{1}{r}\dfrac{d}{dt} \int_{\Omega}w_\epsilon^{\frac{r}{2}}\,dx
& + C_0\int_{\Omega}w_\epsilon^{\frac{p+r}{2}-2}|u_{xx}|^2\,dx
 + C_1\int_\Omega \left|\nabla \left(w_\epsilon^{\frac{p+r-2}{4}}\right)\right|^2\,dx
\\
&
\leq \int_\Omega |f|\left|\div\left(w_\epsilon^{\frac{r-2}{2}} \nabla u\right)\right|\,dx +C \equiv\mathcal{J} +C
\end{split}
\end{equation}
with constants $C$, $C_0$, $C_1$ depending on $r$, $p^\pm$, $N$, $\partial \Omega$, $\|u(t)\|_{2,\Omega}$. By Young's inequality

\[
\begin{split}
\mathcal{J} & = \int_{\Omega}|f|\left |w_\epsilon^{\frac{r-2}{2}} \Delta u + (r-2) w_\epsilon^{\frac{r-2}{2} -1} \sum_{i,j =1}^N D_iuD_juD_{ij}^2u\right|\,dx
\\
    & \leq C \int_{\Omega} |f| w_\epsilon^{\frac{r-2}{2}} |u_{xx}| ~dx + (r-2) \int_{\Omega} |f| w_\epsilon^{\frac{r-2}{2}-1} |\nabla u|^2 |u_{xx}| ~dx
    \\
    & \leq C (r-1) \int_{\Omega} \left(|f| w_\epsilon^{\frac{r-2}{2} - \frac{r+p}{4}+1} \right) \left(w_\epsilon^{\frac{r+p}{4}-1} |u_{xx}| \right) ~dx
    \\
    & \leq \delta \int_{\Omega} w_\epsilon^{\frac{r+p}{2}-2} |u_{xx}|^2~dx + C_\delta r^2 \int_{\Omega} f^2  w_\epsilon^{\frac{r-p}{2}} ~dx
\end{split}
\]
with any $\delta>0$. For the sufficiently small $\delta$ the first term on the right-hand side is absorbed in the left-hand side, and \eqref{eq:a-priori-1} is continued as

\begin{equation}
\label{eq:a-priori-5}
\begin{split}
\frac{1}{r}\dfrac{d}{dt} \int_{\Omega}w_\epsilon^{\frac{r}{2}}\,dx
& + C_0
\int_{\Omega}w_\epsilon^{\frac{p+r}{2}-2}|u_{xx}|^2\,dx
+ C_1\int_\Omega \left|\nabla \left(w_\epsilon^{\frac{p+r-2}{4}}\right)\right|^2\,dx
\\
& \leq C_2 +C_3r^2\int_{\Omega}f^2 w_\epsilon^{\frac{r-p}{2}}\,dx %\equiv C+C'r^2\Phi
\end{split}
\end{equation}
with known constants $C_i$ depending on $r$, $p^\pm$, $N$, $\partial \Omega$, and $C_2$ depending also on $\|u(t)\|_{2,\Omega}$.
Integrating \eqref{eq:a-priori-5} in $t$ we obtain

\begin{equation}
\label{eq:ODI-new-1}
\begin{split}
\frac{1}{r}\sup_{(0,T)} \int_{\Omega}w_\epsilon^{\frac{r}{2}}\,dx  & + C_0\int_{Q_T} \left|\nabla \left(w_\epsilon^{\frac{p+r-2}{4}}\right)\right|^2\,dz + C_1\int_{Q_T} w_\epsilon^{\frac{p+r}{2}-2}|u_{xx}|^2\,dz
\\
&
\leq C_2T+ \frac{1}{r}\int_{\Omega}w_{\epsilon}^{\frac{r}{2}}(x,0)\,dx + C_3r^2\int_{Q_T}f^2w_\epsilon^{\frac{r-p}{2}}\,dz.
\end{split}
\end{equation}
It remains to estimate

\begin{equation}
\label{eq:integral-I}
\mathcal{I}=\int_{Q_T}f^2 w_\epsilon^{\frac{r-p}{2}}\,dz.
\end{equation}
For every $\sigma>2$

\[
\mathcal{I}\leq \|f\|_{\sigma,Q_T}^2 \left(\int_{Q_T}w_{\epsilon}^{\frac{(r-p)\sigma}{2(\sigma-2)}}\,dz\right)^{\frac{\sigma-2}{\sigma}}.
\]
Set $\sigma=N+2$ and observe that since $p>\frac{2N}{N+2}$ by assumption, then

\[
\frac{(r-p)\sigma}{2(\sigma-2)}=\frac{(r-p)(N+2)}{2N}<r\dfrac{N+2}{2N}-1 =r\left(\frac{1}{2}+\frac{1}{N}\right)-1<\frac{p+r-2+\frac{2r}{N}}{2}.
\]
By the Young inequality, for every $\mu>0$

\begin{equation}
  \label{eq:I-1}
  \mathcal{I}\leq \|f\|_{N+2,Q_T}^2 \left(C+ \mu \int_{Q_T}w_{\epsilon}^{\frac{p+r-2+\frac{2r}{N}}{2}}\,dz\right)^{\frac{N}{N+2}},\qquad C=C(\mu).
  \end{equation}
By the Sobolev embedding $W^{1,\frac{2N}{N+2}}(\Omega)
\subset L^2(\Omega)$ and for every $v\in W^{1,\frac{2N}{N+2}}(\Omega)$

\begin{equation}
\label{eq:Sob-1}
\|v\|_{2,\Omega}^2\leq C_{s}\left(\|\nabla v\|_{\frac{2N}{N+2},\Omega}^2 + \|v\|^2_{\frac{2N}{N+2},\Omega}\right)
\end{equation}
with an independent of $v$ constant $C_s$. Applying \eqref{eq:Sob-1} to $w_\epsilon^{\frac{p+r-2+\frac{2r}{N}}{4}}$ we arrive at the inequality

\[
\begin{split}
\int_{\Omega}w_\epsilon^{\frac{p+r-2+\frac{2r}{N}}{2}}\,dx \leq C_s'\left(\int_\Omega \left|\nabla \left(w_\epsilon^{\frac{p+r-2+\frac{2r}{N}}{4}}\right)\right|^{\frac{2N}{N+2}} \,dz\right)^{\frac{N+2}{N}} + C_s' \left(\int_\Omega \left(w_\epsilon^{\frac{p+r-2+\frac{2r}{N}}{4}}\right)^{\frac{2N}{N+2}} \,dz\right)^{\frac{N+2}{N}},
\end{split}
\]
with $C'_s$ depending on $C_s$ and $N$. By the straightforward computation

\[
\begin{split}
D_i\left(w_\epsilon^{\frac{p+r-2+\frac{2r}{N}}{4}}\right)=\frac{1}{4}  w_\epsilon^{\frac{p+r-2+\frac{2r}{N}}{4}}\ln w_\epsilon D_ip + \frac{p+r-2+\frac{2r}{N}}{2}w_{\epsilon}^{\frac{p+r-2+\frac{2r}{N}}{4}-1}\sum_{j=1}^N
D_{j}uD_{ij}^2u.
\end{split}
\]
It follows that

\[
\begin{split}
\int_{\Omega}w_\epsilon^{\frac{p+r-2+\frac{2r}{N}}{2}} & \,dx \leq C_s'' \left(r^2\int_\Omega \left( w_\epsilon^{\frac{p+r-4+\frac{2r}{N}}{4}}|u_{xx}|\right)^{\frac{2N}{N+2}} \,dx\right)^{\frac{N+2}{N}}
\\
& + C_s'' \left(\int_\Omega \left( w_\epsilon^{\frac{p+r-2+\frac{2r}{N}}{4}}|\ln w_\epsilon|\right)^{\frac{2N}{N+2}}\,dx\right)^{\frac{N+2}{N}}
%\\
%&
+C_s''\left(\int_\Omega \left( w_\epsilon^{\frac{p+r-2+\frac{2r}{N}}{4}}\right)^{\frac{2N}{N+2}}\,dx\right)^{\frac{N+2}{N}}
\\
& \equiv C_s''\left(r^2I_1+I_2+I_3\right)
\end{split}
\]
with a constant $C_s''$ depending on $C_s'$, $p^\pm$, $L_p$, $N$, and $r$, but independent of $u$ and $w_\epsilon$. The integrals $I_k$ are estimated separately. By H\"older's inequality with the conjugate exponents $\frac{N+2}{N}$ and $\frac{N+2}{2}$

\[
\begin{split}
I_1^{\frac{N}{N+2}} &  \equiv \int_{\Omega}\left(w_{\epsilon}^{\frac{p+r-4}{2}}|u_{xx}|^2\right)^{\frac{N}{N+2}} \left(w_{\epsilon}^{\frac{r}{N+2}}\right)\,dx \leq \left(\int_{\Omega} w_{\epsilon}^{\frac{p+r-4}{2}}|u_{xx}|^2\,dx\right)^{\frac{N}{N+2}} \left(\int_\Omega w_\epsilon^{\frac{r}{2}}\,dx\right)^{\frac{2}{N+2}},
\end{split}
\]
whence

\[
I_1\leq \left(\int_{\Omega} w_{\epsilon}^{\frac{p+r-4}{2}}|u_{xx}|^2\,dx \right) \left(\int_\Omega w_\epsilon^{\frac{r}{2}}\,dx\right)^{\frac{2}{N}}.
\]
Proceeding in the same way we estimate

\[
\begin{split}
I_2 & \equiv \left(\int_\Omega \left( w_\epsilon^{\frac{p+r-2}{4}}|\ln w_\epsilon|\right)^{\frac{2N}{N+2}}\left(w_\epsilon^{\frac{r}{N+2}}\right) \,dx\right)^{\frac{N+2}{N}}
%\\
%&
\leq \int_\Omega  w_\epsilon^{\frac{p+r-2}{2}}\ln^2 w_\epsilon\,dx \left(\int_\Omega w_\epsilon^{\frac{r}{2}}\,dx\right)^{\frac{2}{N}},
\end{split}
\]

\[
I_3 \equiv \int_\Omega \left( w_\epsilon^{\frac{p+r-2}{4}}\right)^{\frac{2N}{N+2}}\left(w_\epsilon^{\frac{r}{N+2}} \right)\,dx\leq \int_\Omega  w_\epsilon^{\frac{p+r-2}{2}}\,dx \left(\int_\Omega w_\epsilon^{\frac{r}{2}}\,dx\right)^{\frac{2}{N}}.
\]
Gathering the estimates on $I_k$ we obtain the inequality

\begin{equation}
\label{eq:new-1}
\int_{\Omega} w_\epsilon^{\frac{p+r-2+\frac{2r}{N}}{2}}\,dx\leq \widetilde{C} \Pi_1(t) \Pi_2^{\frac{2}{N}}(t),
\end{equation}
where

\[
\begin{split}
& \Pi_1(t) = r^2\int_\Omega w_\epsilon^{\frac{p+r-4}{2}}|u_{xx}|^2\,dx+ \int_{\Omega}w_\epsilon^{\frac{p+r-2}{2}}\left(1+\ln^2w_\epsilon\right)\,dx,
\\
& \Pi_2(t)= \int_\Omega w_\epsilon^{\frac{r}{2}}\,dx,
\end{split}
\]
and $\widetilde C$ is a constant independent of $u$ and $w_\epsilon$. It follows that

\begin{equation}
\label{eq:new-2}
\int_{Q_T} w_\epsilon^{\frac{p+r-2+\frac{2r}{N}}{2}}\,dx\leq \widetilde{C} \left(\sup_{(0,T)}\Pi_2(t)\right)^{\frac{2}{N}}\int_{Q_T}\Pi_1(t)\,dt.
\end{equation}
Now we plug the obtained inequalities into \eqref{eq:I-1}:

\[
\begin{split}
\mathcal{I} & \leq \|f\|^2_{N+2,Q_T}\left(C(\mu)+\mu C'\left(\sup_{(0,T)}\Pi_2(t)\right)^{\frac{2}{N}}\int_{Q_T}\Pi_1(t)\,dt \right)^{\frac{N}{N+2}}
\\
& \leq C''\|f\|^2_{N+2,Q_T}\left(C(\mu)+\mu C' \left(\sup_{(0,T)}\Pi_2(t)\right)^{\frac{2}{N+2}}\left(\int_{Q_T}\Pi_1(t) \,dt\right)^{\frac{N}{N+2}} \right)
\end{split}
\]
with an arbitrary $\mu>0$ and constants $C(\mu)$, $C'$, $C''$ independent of $w_\epsilon$ and $u$. By Young's inequality with the conjugate exponents $\frac{N+2}{N}$ and $\frac{N+2}{2}$ we continue the last inequality as follows:

\[
\begin{split}
\mathcal{I} & \leq C''\|f\|^2_{N+2,Q_T}\left(C(\mu)+ C'\left(\mu^{\frac{N+2}{4}} \sup_{(0,T)}\Pi_2(t)\right)^{\frac{2}{N+2}} \left(\mu^{\frac{N+2}{2N}}\int_{Q_T}\Pi_1(t)\,dt\right)^{\frac{N}{N+2}}
\right)
\\
& \leq  C''\|f\|^2_{N+2,Q_T}\left(C(\mu)+ C'\mu^{\frac{N+2}{4}} \sup_{(0,T)}\Pi_2(t)+ C'  \mu^{\frac{N+2}{2N}}\int_{Q_T}\Pi_1(t)\,dt\right)
\\
& \equiv C''\|f\|^2_{N+2,Q_T}\left(C(\mu)+ \mu^{\frac{N+2}{4}} \sup_{(0,T)}\|w_\epsilon\|_{\frac{r}{2},\Omega}^{\frac{r}{2}}\right.
\\
& \qquad \left.
+C'\mu^{\frac{N+2}{2N}}\left[r^2\int_{Q_T} w_\epsilon^{\frac{p+r-4}{2}}|u_{xx}|^2\,dz+\int_{Q_T} w_\epsilon^{\frac{p+r-2}{2}} \left(1+\ln^2w_\epsilon\right)\,dz\right]\right).
\end{split}
\]
By Lemma \ref{le:racsam} and \eqref{eq:elem}, the second integral in the square brackets is bounded by

\[
\nu \int_{Q_T}w_\epsilon^{\frac{p+r-4}{2}}|u_{xx}|^2\,dz + \widehat C
\]
with an arbitrary $\nu>0$ and a constant $\widehat C=\widehat C(r,N,p^\pm,r,L_p,\nu,\|u\|_{2,\Omega})$. Thus,

\begin{equation}
\label{eq:I-2}
\mathcal{I}\leq C_1 \|f\|_{N+2,Q_T}^2\left( \mu^{\frac{N+2}{4}} \sup_{(0,T)}\|w_\epsilon\|_{\frac{r}{2},\Omega}^{\frac{r}{2}} + \mu^{\frac{N+2}{2N}} \left(r^2+\nu\right) \int_{Q_T}w_\epsilon^{\frac{p+r-4}{2}}|u_{xx}|^2\,dz + C_2\right)
\end{equation}
with a constant $C_1$ depending only on $\textbf{data}$, and $C_2$ depending on $\textbf{data}$ and $\mu$, $\nu$. Substituting \eqref{eq:I-2} into \eqref{eq:ODI-new-1}, using \eqref{eq:interpol-1} and choosing $\mu$ and $\nu$ sufficiently small we transform \eqref{eq:ODI-new-1} into

\begin{equation}
\label{eq:final}
\begin{split}
\sup_{(0,T)} \int_{\Omega}w_\epsilon^{\frac{r}{2}}\,dx  & + \alpha\int_{Q_T} \left|\nabla \left(w_\epsilon^{\frac{p+r-2}{4}}\right)\right|^2\,dz + \beta \int_{Q_T} w_\epsilon^{\frac{p+r+\rho-4}{2}}|\nabla u|^2 \,dz
%\\
%&
\leq \gamma + \int_{\Omega}w_{\epsilon}^{\frac{r}{2}}(x,0)\,dx
\end{split}
\end{equation}
with any $0<\rho<\frac{4}{N+2}$ and finite positive constants $\alpha$, $\beta$, $\gamma$ depending only on $\textbf{data}$, $r$, $\rho$, and $\|f\|_{N+2,Q_T}$ but independent of $\epsilon$.

\subsection{Equation without nonlinear sources - II: the case $\sigma \in (2,N+2)$}
\label{subsec:est-2}
We refine the estimates on the integral \eqref{eq:integral-I} and extend them to the case of low integrability of $f$.
Let us take $\sigma>2$, $r\geq \max\{2,p^+\}$, and search for a (optimal) $\beta$ such that the following inequality holds true:
\begin{equation}\label{imp:ineq:modi}
    \frac{(r-p)\sigma}{2(\sigma-2)} \leq \frac{p+r-2+\frac{2r}{\beta}}{2}.
\end{equation}
By the Young inequality (cf. with \eqref{eq:I-1})

\begin{equation}
  \label{eq:I-1-modi}
  \mathcal{I}\leq \|f\|_{\sigma,Q_T}^2 \left(C+ \int_{Q_T}w_{\epsilon}^{\frac{p+r-2+\frac{2r}{\beta}}{2}}\,dz\right)^{\frac{\sigma-2}{\sigma}}.
  \end{equation}
Fix

\[
\mu \in \left(\frac{N}{N+2},1\right) \qquad \text{and set} \qquad \alpha_\sharp = 2 \mu, \quad  \quad \alpha_\sharp^\ast = \dfrac{2\mu N}{N-2\mu}> 2.
\]
By the Sobolev embedding, $W^{1,\alpha_\sharp}(\Omega)
\subset L^{\alpha_\sharp^\ast}(\Omega) \subset L^2(\Omega)$, and for every $v\in W^{1, \alpha_\sharp}(\Omega)$

\begin{equation}
\label{eq:Sob-1-modi}
\|v\|_{2,\Omega}^{2} \leq C_{s}\left(\|\nabla v\|_{\alpha_\sharp,\Omega}^{2} + \|v\|^{2}_{\alpha_\sharp,\Omega}\right)
\end{equation}
with an independent of $v$ constant $C_s$. Applying \eqref{eq:Sob-1-modi} to $w_\epsilon^{\frac{p+r-2+\frac{2r}{\beta}}{4}} $ we arrive at the inequality

\[
\begin{split}
\int_{\Omega}w_\epsilon^{\frac{p+r-2+\frac{2r}{\beta}}{2}}\,dx \leq C_s'\left(\int_\Omega \left|\nabla \left(w_\epsilon^{\frac{p+r-2+\frac{2r}{\beta}}{4}}\right)\right|^{2 \mu}\,dz \right)^{\frac{1}{\mu}} + C_s' \left(\int_\Omega \left(w_\epsilon^{\frac{p+r-2+\frac{2r}{\beta}}{4}}\right)^{2 \mu}\,dz \right)^{\frac{1}{\mu}},
\end{split}
\]
with $C'_s$ depending on $C_s$ and $\sigma$. By the straightforward computation

\[
\begin{split}
D_i\left(w_\epsilon^{\frac{p+r-2+\frac{2r}{\beta}}{4}}\right)=\frac{1}{4}  w_\epsilon^{\frac{p+r-2+\frac{2r}{\beta}}{4}}\ln w_\epsilon D_ip + \frac{p+r-2+\frac{2r}{\beta}}{4} w_{\epsilon}^{\frac{p+r-2+\frac{2r}{\beta}}{4}-1}\sum_{j=1}^N
D_{j}uD_{ij}^2u.
\end{split}
\]
It follows that

\[
\begin{split}
\int_{\Omega}w_\epsilon^{\frac{p+r-2+\frac{2r}{\beta}}{2}} \,dx & \leq C_s'' \left(r^2\int_\Omega \left( w_\epsilon^{\frac{p+r-4+\frac{2r}{\beta}}{4}}|u_{xx}|\right)^{2 \mu}\,dx \right)^{\frac{1}{\mu}}
\\
& \qquad  + C_s'' \left(\int_\Omega \left( w_\epsilon^{\frac{p+r-2+\frac{2r}{\beta}}{4}}|\ln w_\epsilon|\right)^{2 \mu}\,dx\right)^{\frac{1}{\mu}}
+C_s''\left(\int_\Omega \left( w_\epsilon^{\frac{p+r-2+\frac{2r}{\beta}}{4}}\right)^{2 \mu}\,dx \right)^{\frac{1}{\mu}}
\\
& \equiv C_s''\left(r^2I_1+I_2+I_3\right)
\end{split}
\]
with a constant $C_s''$ depending on $C_s'$, $p^\pm$, $L_p$, $N$, and $r$, but independent of $u$ and $w_\epsilon$. The integrals $I_k$ are estimated separately. By H\"older's inequality with the conjugate exponents $\frac{2}{\alpha_\sharp} =\frac{1}{\mu} >1$ and $ \frac{1}{1-\mu}$

\[
\begin{split}
I_1^{\mu} &  \equiv \int_{\Omega}\left(w_{\epsilon}^{\frac{p+r-4}{2}}|u_{xx}|^2\right)^{\mu} \left(w_{\epsilon}^{\frac{r}{2}}\right)^{\frac{2\mu}{\beta}}\,dx \leq \left(\int_{\Omega} w_{\epsilon}^{\frac{p+r-4}{2}}|u_{xx}|^2\,dx\right)^{\mu} \left(\int_\Omega \left(w_\epsilon^{\frac{r}{2}}\right)^{\frac{2\mu }{\beta(1-\mu)}}\,dx\right)^{(1-\mu)}.
\end{split}
\]
Let

\[
\beta=\frac{2\mu}{1-\mu}.
\]
With this choice of $\beta$ we may estimate

\[
I_1\leq \left(\int_{\Omega} w_{\epsilon}^{\frac{p+r-4}{2}}|u_{xx}|^2\,dx \right) \left(\int_\Omega w_\epsilon^{\frac{r}{2}}\,dx\right)^{\frac{(1-\mu)}{\mu}},
\]
while inequality \eqref{imp:ineq:modi} takes on the form

\begin{equation}
\notag
     \begin{split}
     \frac{(r-p)\sigma}{2(\sigma-2)} & \leq \frac{p+r-2+\frac{r(1-\mu)}{\mu}}{2} \qquad  \Leftrightarrow \quad
      \mu \sigma (r-p)  \leq  \mu (p+r-2)(\sigma-2) + r(1-\mu) (\sigma-2).
     \end{split}
     \end{equation}

Solving the last inequality for $r$ and gathering the result with the conditions of Proposition \ref{pro:pointwise} and using the fact that $\mu > \frac{N}{N+2} > \frac{\sigma-2}{\sigma}$, we arrive at the following restriction on $r$ in terms of $\sigma$, $p^\pm$, and $N$:

\begin{equation}
\label{eq:r-modi}
\max\{p^+,2\}\leq r \leq \frac{2(p^-(\sigma-1) -\sigma + 2)}{\sigma - \frac{\sigma-2}{\mu}}.
\end{equation}

Proceeding in the same way, we estimate

\[
\begin{split}
I_2 & \equiv \left(\int_\Omega \left( w_\epsilon^{\frac{p+r-2}{2}}\ln^2 w_\epsilon\right)^{\mu}\left(w_{\epsilon}^{\frac{r}{2}} \right)^\frac{2\mu}{\beta}\,dx\right)^{\frac{1}{\mu}}
\leq \int_\Omega  w_\epsilon^{\frac{p+r-2}{2}}\ln^{2} w_\epsilon\,dx \left(\int_\Omega w_\epsilon^{\frac{r}{2}}\,dx\right)^{\frac{1-\mu}{\mu}},
\end{split}
\]

\[
I_3 \equiv \left(\int_\Omega \left( w_\epsilon^{\frac{p+r-2}{2}}\right)^{\mu}\left(w_{\epsilon}^{\frac{r}{2}} \right)^\frac{2\mu}{\beta}\,dx\right)^{\frac{1}{\mu}}
\leq \int_\Omega  w_\epsilon^{\frac{p+r-2}{2}}\,dx \left(\int_\Omega w_\epsilon^{\frac{r}{2}}\,dx\right)^{\frac{1-\mu}{\mu}}.
\]

Gathering the estimates on $I_k$ we obtain the inequality

\begin{equation}
\label{eq:new-1-modi}
\int_{\Omega} w_\epsilon^{\frac{p+r-2+\frac{2r}{\beta}}{2}}\,dx\leq \widetilde{C} \Pi_1(t) \Pi_2^{\frac{1-\mu}{\mu}}(t),
\end{equation}
where

\[
\begin{split}
& \Pi_1(t) = r^2\int_\Omega w_\epsilon^{\frac{p+r-4}{2}}|u_{xx}|^2\,dx+ \int_{\Omega}w_\epsilon^{\frac{p+r-2}{2}}\left(1+\ln^2w_\epsilon\right)\,dx,
\\
& \Pi_2(t)= \int_\Omega w_\epsilon^{\frac{r}{2}}\,dx,
\end{split}
\]
and $\widetilde C$ is a constant independent of $u$ and $w_\epsilon$. It follows that

\begin{equation}
\label{eq:new-2-modi}
\int_{Q_T} w_\epsilon^{\frac{p+r-2+\frac{2r}{\beta}}{2}}\,dx\leq \widetilde{C} \left(\sup_{(0,T)}\Pi_2(t)\right)^{\frac{1-\mu}{\mu}}\int_{Q_T}\Pi_1(t)\,dt.
\end{equation}
Now we plug the obtained inequalities into \eqref{eq:I-1-modi}:

\begin{equation}
\label{eq:last-modi}
\begin{split}
\mathcal{I} & \leq \|f\|^2_{\sigma,Q_T}\left(C+ C'\left(\sup_{(0,T)}\Pi_2(t)\right)^{\frac{1-\mu}{\mu}
}\int_{Q_T}\Pi_1(t)\,dt \right)^{\frac{\sigma-2}{\sigma}}
\\
& \leq C''\|f\|^2_{\sigma,Q_T}\left(1+ \left(\sup_{(0,T)}\Pi_2(t)\right)^{\frac{(1-\mu)(\sigma-2)}{\mu \sigma} }\left(\int_{Q_T}\Pi_1(t)\,dt\right)^{\frac{\sigma-2}{\sigma}} \right)
\end{split}
\end{equation}
with constants $C'$, $C''$ independent of $w_\epsilon$ and $u$.

For any $\sigma \in (2, N+2)$ and $1>\mu >  \frac{N}{N+2}$, we have the following inequalities:
\begin{equation}
\label{eq:ineq-app}
\begin{split}
\mu > \frac{N}{N+2} > \frac{\sigma-2}{\sigma} \quad  & \Leftrightarrow \quad 2 \mu - (1-\mu)(\sigma-2) >0 \quad \Leftrightarrow \quad  \frac{\mu(\sigma-2)}{\mu \sigma-(1-\mu)(\sigma-2)}<1
\end{split}
\end{equation}
and
\begin{equation}
     \mu > \frac{N}{N+2} > \frac{(\sigma-2)}{2(\sigma-1)}  \qquad \Leftrightarrow \quad \sigma \mu - (1-\mu)(\sigma-2) >0 \qquad \Leftrightarrow \quad  \frac{(1-\mu)(\sigma-2)}{\mu \sigma} < 1.
\end{equation}

Now, by applying Young's inequality two times with the exponents
\[
\frac{\mu \sigma}{(1-\mu)(\sigma-2)} ,\quad \gamma= \frac{\mu \sigma}{\mu \sigma-(1-\mu)(\sigma-2)}, \quad \text{and}\quad  \frac{\sigma}{\gamma(\sigma-2)},\quad  \frac{\sigma}{\sigma-\gamma(\sigma-2)},
\]

we find that for every $\lambda >0$
\[
\begin{split}
\left(\sup_{(0,T)}\Pi_2(t)\right)^{\frac{(1-\mu)(\sigma-2)}{\mu \sigma}} & \left(\int_{Q_T}\Pi_1(t)\,dt\right)^{\frac{\sigma-2}{\sigma}}
\leq\lambda^\frac{\mu \sigma}{(1-\mu)(\sigma-2)}\sup_{(0,T)}\Pi_2(t) + \frac{1}{\lambda^\gamma} \left(\int_{Q_T}\Pi_1(t)\,dt\right)^{\frac{\gamma(\sigma-2)}{\sigma}}\\
& \leq\lambda^\frac{\mu \sigma}{(1-\mu)(\sigma-2)}\sup_{(0,T)}\Pi_2(t) +  \lambda^\frac{\sigma}{\sigma-2} \int_{Q_T}\Pi_1(t)\,dt + C(\lambda).
\end{split}
\]
Estimate \eqref{eq:last-modi} is then continued as follows: for $\sigma\in (2,N+2)$ and  $\mu \in \left(\frac{N}{N+2}, 1 \right)$,

\[
\begin{split}
\mathcal{I}
& \leq  C''\|f\|^2_{\sigma,Q_T}\left(C(\lambda)+ \lambda^\frac{\mu \sigma}{(1-\mu)(\sigma-2)}
 \sup_{(0,T)}\Pi_2(t)+  \lambda^{\frac{\sigma}{\sigma-2}}\int_{Q_T}\Pi_1(t)\,dt\right)
\\
& \equiv C''\|f\|^2_{\sigma,Q_T}\left(C(\lambda)+ \lambda^\frac{\mu \sigma}{(1-\mu)(\sigma-2)}  \sup_{(0,T)}\|w_\epsilon\|_{\frac{r}{2},\Omega}^{\frac{r}{2}}\right.
\\
& \qquad \qquad \qquad \qquad  \left.
+ C'\lambda^{\frac{\sigma}{(\sigma-2)}}\left[r^2\int_{Q_T} w_\epsilon^{\frac{p+r-4}{2}}|u_{xx}|^2\,dz+\int_{Q_T} w_\epsilon^{\frac{p+r-2}{2}} \left(1+\ln^2w_\epsilon\right)\,dz\right]\right).
\end{split}
\]
By Lemma \ref{le:racsam} and \eqref{eq:elem}, the second integral in the square brackets is bounded by

\[
\nu \int_{Q_T}w_\epsilon^{\frac{p+r-4}{2}}|u_{xx}|^2\,dz + \widehat C
\]
with an arbitrary $\nu>0$ and a constant $\widehat C=\widehat C(r,N,p^\pm,r,L_p,\nu,\|u\|_{2,\Omega})$. Thus,

\begin{equation}
\label{eq:I-2-new}
\mathcal{I}\leq C_1 \|f\|_{\sigma,Q_T}^2\left(\lambda^\frac{\mu \sigma}{(1-\mu)(\sigma-2)}  \sup_{(0,T)}\|w_\epsilon\|_{\frac{r}{2},\Omega}^{\frac{r}{2}} + \lambda^{\frac{\sigma}{\sigma-2}}\left(r^2+\nu\right) \int_{Q_T}w_\epsilon^{\frac{p+r-4}{2}}|u_{xx}|^2\,dz + C_2\right)
\end{equation}
with a constant $C_1$ depending only on $\textbf{data}$, and $C_2$ depending on $\textbf{data}$ and $\mu$, $\nu$. Substituting \eqref{eq:I-2-new} into \eqref{eq:ODI-new-1}, using \eqref{eq:interpol-1} and choosing $\lambda$ and $\nu$ sufficiently small we transform \eqref{eq:ODI-new-1} into

\begin{equation}
\label{eq:final-new}
\begin{split}
\sup_{(0,T)} \int_{\Omega}w_\epsilon^{\frac{r}{2}}\,dx  & + \alpha\int_{Q_T} \left|\nabla \left(w_\epsilon^{\frac{p+r-2}{4}}\right)\right|^2\,dz + \beta \int_{Q_T} w_\epsilon^{\frac{p+r+\rho-4}{2}}|\nabla u|^2 \,dz
\leq \gamma + \int_{\Omega}w_{\epsilon}^{\frac{r}{2}}(x,0)\,dx
\end{split}
\end{equation}
with any $0<\rho<\frac{4}{N+2}$ and finite positive constants $\alpha$, $\beta$, $\gamma$ depending only on $\textbf{data}$, $r$, $\rho$, and $\|f\|_{\sigma,Q_T}$ but independent of $\epsilon$. Since $\mu \in \left(\frac{N}{N+2},1\right)$ is arbitrary, condition \eqref{eq:filter} allows one to choose $\mu$ so close to $\frac{N}{N+2}$ that \eqref{eq:r-modi} holds true.

\subsection{Equation with nonlinear sources}

Assume first that $f\in L^{N+2}(Q_T)$. Let $u$ be the classical solution of problem \eqref{eq:main-reg} corresponding to $\textbf{data}$, and the nonlinear source $F_\epsilon(z,u,\nabla u)$ defined in \eqref{eq:sources}. By Proposition \ref{pro:existence-smooth-data} the solution $u(z)$ satisfies estimate \eqref{eq:unif-reg}. Multiplying \eqref{eq:main-reg} by $\operatorname{div}\left(w_\epsilon^{\frac{r-2}{2}}\nabla u\right)$ and following the derivation of \eqref{eq:a-priori-1}, \eqref{eq:a-priori-5}, we arrive at the inequality

\[
\begin{split}
\frac{1}{r}\dfrac{d}{dt} \int_{\Omega}w_\epsilon^{\frac{r}{2}}\,dx
& + C_0
\int_{\Omega}w_\epsilon^{\frac{p+r}{2}-2}|u_{xx}|^2\,dx
+ C_1\int_\Omega \left|\nabla \left(w_\epsilon^{\frac{p+r-2}{4}}\right)\right|^2\,dx
\\
& \leq C_2 +C_3r^2\int_{\Omega}(f^2+F_{1\epsilon}^2)w_\epsilon^{\frac{r-p}{2}}\,dx +\int_\Omega F_{2\epsilon}\operatorname{div}\left(w_\epsilon^{\frac{r-2}{2}}\nabla u \right)\,dx.
\end{split}
\]
Integrating in $t$, by analogy with \eqref{eq:I-1} we find that for every $t\in (0,T)$

\begin{equation}
\label{eq:ODI-new-4-prim-prim}
\begin{split}
\frac{1}{r}\int_{\Omega} & w_\epsilon^{\frac{r}{2}}(t)\,dx + C_1\int_{Q_t} \left|\nabla \left(w_\epsilon^{\frac{p+r-2}{4}}\right)\right|^2\,dz + C_2\int_{Q_t} w_\epsilon^{\frac{p+r}{2}-2}|u_{xx}|^2\,dz
\\
&
\leq
C_3r^2\left(\|f\|^2_{N+2,Q_t} + \|F_{1\epsilon}\|_{N+2,Q_t}^2\right) \left(C+ \lambda  \int_{Q_T}w_{\epsilon}^{\frac{p+r-2+\frac{2r}{N}}{2}}\,dz\right)^{\frac{N}{N+2}}
\\
& + C_4 \left|\int_{Q_t}F_{2\epsilon}\operatorname{div}\left(w_\epsilon^{\frac{r-2}{2}}\nabla u \right)\,dz\right| + C_2
+ \frac{1}{r}\int_{\Omega}w_{\epsilon}^{\frac{r}{2}}(x,0)\,dx
\end{split}
\end{equation}
with any $\lambda\in (0,1)$ and $r\geq  \max\{p^+,2\}$. To estimate $\|F_{1\epsilon}\|_{N+2,Q_t}^2$ we use the Young inequality,

\[
\|F_{1\epsilon}\|^2_{N+2,Q_t} \leq C\left(\int_{Q_t}(\epsilon^2+u^2)^{(N+2)\frac{q-1}{2}}\,dz\right)^{\frac{2}{N+2}} \leq C_1\left(1 + \|u\|_{(N+2)(q^+-1),Q_t}^{2(q^+-1)}\right),
\]
and then apply the parabolic embedding inequality \cite[Ch.I, Proposition 3.2]{DB}:

\begin{equation}
\label{eq:DB}
 \|u\|_{(N+2)(q^+-1),Q_t}^{q^+-1}\leq C\left(\|\nabla u\|_{p^-,Q_t}^{q^+-1}+\sup_{(0,T)}\|u(t)\|^{q^+-1}_{2,\Omega}\right),
\end{equation}
provided that

\[
(N+2)(q^+-1)\leq p^-\dfrac{N+2}{N}\qquad \Leftrightarrow \qquad q^+\leq 1+ \dfrac{p^-}{N}.
\]
By virtue of \eqref{eq:unif-reg}, both terms on the right-hand side of \eqref{eq:DB} are bounded by a constant depending only on $\textbf{data}$.

By the Young inequality
\[
\begin{split}
 \int_{Q_t}\left|F_{2\epsilon}(z,\nabla u)\operatorname{div}\left(w_\epsilon^{\frac{r-2}{2}}\nabla u\right)\right|\,dz & \leq C \int_{Q_t} w_\epsilon^{\frac{s+r-3}{2}}|u_{xx}|\,dz
 \\
 & =C\int_{Q_t} w_{\epsilon}^{\frac{2s+r-p-2}{4}} \left(w_\epsilon^{\frac{p+r-4}{2}}|u_{xx}|^2\right)^{\frac{1}{2}}\,dz
\\
& \leq \delta \int_{Q_t}w_\epsilon^{\frac{p+r-4}{2}}|u_{xx}|^2\,dz + C' \int_{Q_t} w_{\epsilon}^{\frac{2s+r-p-2}{2}}\,dz
\end{split}
\]
with any $\delta>0$. Since $s(z)\leq p(z)$, $r\geq 2$, and $p^->\frac{2N}{N+2}$ by assumption, there is $\lambda\in \left(0,\frac{4}{N+2}\right)$ such that $p+r+\lambda>4$ in $Q_T$. We notice that (cf. with \eqref{eq:w-u})
\[
\begin{split}
w_{\epsilon}^{\frac{2s+r-p-2}{2}} & \leq 1+  w_{\epsilon}^{\frac{p+r+\lambda-2}{2}} \leq 1 + \begin{cases}
2w_{\epsilon}^{\frac{p+r+\lambda-4}{2}}|\nabla u|^2 & \text{if $|\nabla u|\geq \epsilon$},
\\
(2\epsilon^2)^{\frac{p+r+\lambda-2}{2}} & \text{otherwise}
\end{cases}
\\
& \leq 2\left(1+ w_{\epsilon}^{\frac{p+r+\lambda-4}{2}}|\nabla u|^2\right)
\end{split}
\]
and apply \eqref{eq:interpol-1}: for every $\delta>0$
\[
\int_{Q_T}w_{\epsilon}^{\frac{2s+r-p-2}{2}}\,dz\leq 2T|\Omega| + \int_{Q_T}w_{\epsilon}^{\frac{p+r+\lambda-4}{2}}|\nabla u|^2\,dz\leq \delta \int_{Q_T}w_\epsilon^{\frac{p+r}{2}-2}|u_{xx}|^2\,dz +C
\]
with a constant $C$ depending only on $\textbf{data}$ and $r$. Choosing $\delta$ sufficiently small, we absorb the last integral on the right-hand side of \eqref{eq:ODI-new-4-prim-prim} in the left-hand side and arrive at inequality \eqref{eq:final}. The above arguments are summarized in the following assertion.

\begin{lemma}
\label{le:est-full}
Let $u(z)$ be a classical solution of problem \eqref{eq:main-reg} with $F_\epsilon\not\equiv 0$ and smooth $\textbf{data}_m$. If the coefficients $a$, $\vec c$ and the exponents $p$, $q$, $s$ satisfy conditions \eqref{eq:data-0}, \eqref{eq:data}, $f\in L^{N+2}(Q_T)$, and

\[
q^+\leq 1+\frac{p^-}{N},\qquad s^+\leq p^-,
\]
then for every $r\geq \max \{p^+,2\}$

\[
\sup_{(0,T)}\|\nabla u(t)\|_{r,\Omega}^r\leq \sup_{(0,T)}\|w_\epsilon(t)\|^{\frac{r}{2}}_{\frac{r}{2},\Omega}\leq C+ \|\nabla u_0\|_{r,\Omega}^{r}
\]
with a constant $C$ depending on $\textbf{data}$, $L_p$, $p^\pm$, $q^\pm$, $s^\pm$, $N$, $r$ but independent of $\epsilon$. Moreover,

\begin{equation}
\label{eq:final-2}
\int_{Q_T}  w_\epsilon^{\frac{p+r+\rho-4}{2}}|\nabla u|^2 \,dz  +\int_{Q_T}\left|\nabla \left(w_\epsilon^{\frac{p(z)+r-2}{4}}\right)\right|^2\,dz\leq C'
\end{equation}
with a constant $C'$ depending on the same quantities as $C$ and $\rho$.
\end{lemma}

Let $f\in L^{\sigma}(Q_T)$ with $\sigma \in (2,N+2)$. Proceeding exactly as in the case $\sigma=N+2$, we arrive at inequality \eqref{eq:ODI-new-4-prim-prim} in which the second term on the right-hand side is now substituted by

\begin{equation}
\label{eq:ODI-new-4-prim-prim-prim}
\begin{split}
& C_3r^2\left(\|f\|^2_{\sigma,Q_t} + \|F_{1\epsilon}\|_{\sigma,Q_t}^2\right) \left(C+ \lambda  \int_{Q_T}w_{\epsilon}^{\frac{p+r-2+\frac{2r}{\beta}}{2}}\,dz\right)^{\frac{\sigma-2}{\sigma}}
\end{split}
\end{equation}
with $\beta=\frac{2\mu}{1-\mu}$ and $r$ satisfying \eqref{eq:filter}. In \eqref{eq:ODI-new-4-prim-prim-prim}, the term including $\|f\|_{\sigma,Q_t}^2$ is estimated in Subsection \ref{subsec:est-2}. The second term is estimated with the help of the parabolic embedding inequality \eqref{eq:DB}:

\[
\begin{split}
\|F_{1\epsilon}\|^2_{\sigma,Q_t}  & \leq C\left(\int_{Q_t}(\epsilon^2+u^2)^{\sigma\frac{q-1}{2}}\,dz\right)^{\frac{2}{\sigma}} \leq C\left(1 + \|u\|_{\sigma(q^+-1),Q_t}^{2(q^+-1)}\right)
\\
& \leq C\left(1+\|\nabla u\|_{p^-,Q_t}^{q^+-1}+\sup_{(0,T)}\|u(t)\|^{q^+-1}_{2,\Omega}\right),
\end{split}
\]
provided that

\[
\sigma(q^+-1)\leq p^-\frac{N+2}{N}\qquad \Leftrightarrow \qquad q^+\leq 1+ p^-\frac{N+2}{\sigma N}.
\]
By virtue of \eqref{eq:unif-reg} $\|F_{1\epsilon}\|_{\sigma,Q_t}^2\leq C(\textbf{data})$. The term involving $F_{2\epsilon}$ is estimated as in the case $f\in L^{N+2}(Q_T)$.

\begin{lemma}
\label{le:est-full-1}
Let $u(z)$ be a classical solution of problem \eqref{eq:main-reg} with $F_\epsilon\not\equiv 0$ and smooth $\textbf{data}_m$. Assume that the coefficients $a$, $\vec c$ and the exponents $p$, $q$, $s$ satisfy conditions \eqref{eq:data-0}, \eqref{eq:data}, $f\in L^{\sigma}(Q_T)$ with $\sigma \in (2,N+2)$, and

\[
q^+\leq 1+p^-\frac{N+2}{\sigma N},\qquad s^+\leq p^-.
\]
For every $r\geq \max \{p^+,2\}$ satisfying condition \eqref{eq:filter}

\[
\sup_{(0,T)}\|\nabla u(t)\|_{r,\Omega}^r\leq \sup_{(0,T)}\|w_\epsilon(t)\|^{\frac{r}{2}}_{\frac{r}{2},\Omega}\leq C+ \|\nabla u_0\|_{r,\Omega}^{r}
\]
with a constant $C$ depending on $\textbf{data}$, $L_p$, $p^\pm$, $q^\pm$, $s^\pm$, $N$, $r$ but independent of $\epsilon$. Moreover,

\begin{equation}
\label{eq:final-2-non}
\int_{Q_T}  w_\epsilon^{\frac{p+r+\rho-4}{2}}|\nabla u|^2 \,dz  +\int_{Q_T}\left|\nabla \left(w_\epsilon^{\frac{p(z)+r-2}{4}}\right)\right|^2\,dz\leq C'
\end{equation}
with a constant $C'$ depending on the same quantities as $C$, and $\rho$.
\end{lemma}

\section{Proofs of the main results}\label{sec:proof-results}
\subsection{Proof of Theorems \ref{th:main-1}, \ref{th:main-1-1}}
We give the detailed proof of Theorem \ref{th:main-1}. The assertion of Theorem \ref{th:main-1-1} follows by the same argument and is therefore omitted.
Let $\{u_{\epsilon m}\}$ be a sequence of classical solutions of problem \eqref{eq:main-reg} with the data $\textbf{data}_m$ and $u_\epsilon=\lim_{m\to \infty}u_{\epsilon m}$ be the solution of problem \eqref{eq:main-reg} with $\textbf{data}$. By Proposition \ref{pro:existence-smooth-data} and Lemma \ref{le:est-full}, for every $r\geq \max\{2,p^+\}$

\begin{equation}
\label{eq:reg-est}
\sup_{(0,T)}\int_{\Omega}(\epsilon^2+|\nabla u_{\epsilon m}|^2)^{\frac{r}{2}}\,dx + \int_{Q_T} (\epsilon^2+|\nabla u_{\epsilon m}|^2)^{\frac{p+r+\rho-4}{2}}|\nabla u_{\epsilon m}|^2 \,dz \leq C
\end{equation}
with a constant $C$ independent of $m$ and $\epsilon$. Letting $m\to \infty$ and using the Fatou lemma we conclude that $u_\epsilon$ satisfy inequality \eqref{eq:reg-est} with the same constant $C$. By Proposition \ref{pro:existence-degenerate} $u=\lim_{\epsilon\to 0}u_\epsilon(z)$ (along a sequence $\{\epsilon_k\}$) is a strong solution of problem \eqref{eq:main}. The functions $w_\epsilon=\epsilon^2+|\nabla u_\epsilon|^2$ converge to $|\nabla u|$ a.e. in $Q_T$ and satisfy the uniform estimate \eqref{eq:reg-est}, whence the conclusion.

\subsection{Proof of Theorem \ref{th:main-2}}
Let $u_{\epsilon m}$ be the solution of problem \eqref{eq:main-reg} and $w_{\epsilon m}=\epsilon^2+|\nabla u_{\epsilon m}|^2$. By \eqref{eq:final-2-non} the classical solutions of problem \eqref{eq:main-reg} with the data $\textbf{data}_m$ satisfy the uniform estimates

\[
\left\|D_i\left(w_{\epsilon m}^{\frac{p_m+r-2}{4}}\right)\right\|_{2,Q_T}\leq C',\qquad i=\overline{1,N}.
\]
It follows that there exist functions $\eta^{(\epsilon)}_i\in L^{2}(Q_T)$ such that

\[
D_i\left(w_{\epsilon m}^{\frac{p_m+r-2}{4}}\right)\rightharpoonup \eta^{(\epsilon)}_i\quad\text{in $L^2(Q_T)$}
\]
(up to a subsequence). To identify $\eta^{(\epsilon)}_i$ we use the pointwise convergence $\nabla u_{\epsilon m}\to \nabla u_\epsilon$ and the uniform convergence $p_m\to p$: for every $\phi\in C_c^\infty(Q_T)$

\begin{equation}
\label{eq:sobolev-der}
\begin{split}
(\eta^{(\epsilon)}_i,\phi)_{2,Q_T} & =
\lim_{m\to \infty}\left(D_i\left(w_{\epsilon m}^{\frac{p_m+r-2}{4}}\right),\phi\right)_{2,Q_T}
\\
& = -\lim_{m\to \infty}\left((\epsilon^2+|\nabla u_{\epsilon m}|^2)^{\frac{p_m+r-2}{4}},D_i\phi\right)_{2,Q_T}
\\
& = - \left( (\epsilon^2+|\nabla u_{\epsilon}|^2)^{\frac{p+r-2}{4}},D_i\phi\right)_{2,Q_T}.
\end{split}
\end{equation}
It follows that $\eta^{(\epsilon)}_i=D_i\left((\epsilon^2+|\nabla u_{\epsilon}|^2)^{\frac{p+r-2}{4}}\right)\in L^2(Q_T)$ and $\|\eta^{(\epsilon)}_i\|_{2,Q_T}$ are uniformly bounded. Then $\eta^{(\epsilon)}_i\rightharpoonup \eta_i$ in $L^2(Q_T)$, and $\eta_i=D_i\left(|\nabla u|^{\frac{p(z)+r-2}{2}}\right)$ because of the pointwise convergence $\nabla u_\epsilon\to \nabla u$. These arguments prove \eqref{eq:main-2} (i).

It is straightforward to compute that for every $i,j=\overline{1,N}$

\[
\begin{split}
D_i\left(w_{\epsilon m}^{\frac{p_m+r}{4}-1}D_ju_{\epsilon m}\right) & = w_{\epsilon m}^{\frac{p_m+r}{4}-1}D_{ij}^2u_{\epsilon m} + \frac{1}{2}(p_m+r-4)w_{\epsilon m}^{\frac{p_m+r}{4}-2}D_ju_{\epsilon m}\sum_{k=1}^ND_ku_{\epsilon m}D_{kj}^2u_{\epsilon m}
\\
&
+ w_{\epsilon m}^{\frac{p_m+r}{4}-1}\ln w_{\epsilon m}D_ip_mD_ju_{\epsilon m}.
\end{split}
\]
By \eqref{eq:elem}

\[
\left|D_i\left(w_{\epsilon m}^{\frac{p_m+r}{4}-1}D_ju_{\epsilon m}\right)\right|\leq C_1 w_{\epsilon m}^{\frac{p_m+r}{4}-1}|\left(u_{\epsilon m}\right)_{xx}|+ C_2w_{\epsilon m}^{\frac{p_m+r-4+\gamma}{4}}|\nabla u_{\epsilon m}| +C_3.
\]
with any $\gamma>0$. Fix $0<\gamma<\frac{4}{N+2}$. By Young's inequality and \eqref{eq:interpol-2}, \eqref{eq:final-2-non}

\[
\begin{split}
\left\|D_i\left(w_{\epsilon m}^{\frac{p_m+r}{4}-1}D_ju_{\epsilon m}\right)\right\|_{2,Q_T}^2 & \leq C_1' \int_{Q_T}w_{\epsilon m}^{\frac{p_m+r}{2}-2}\left|\left(u_{\epsilon m}\right)_{xx}\right|^2\,dz
\\
& + C_2'\int_{Q_T}w_{\epsilon m}^{\frac{p_m+r-4+\gamma}{2}}|\nabla u_{\epsilon m}|^2\,dz + C_3'\leq C
\end{split}
\]
with an independent of $\epsilon$ and $m$ constant $C$. It follows that there exists $\eta_{ij}^{(\epsilon)}\in L^2(Q_T)$ such that

\[
D_i\left(w_{\epsilon m}^{\frac{p_m+r}{4}-1}D_ju_{\epsilon m}\right)\rightharpoonup\eta_{ij}^{(\epsilon)}\quad \text{in $L^2(Q_T)$},
\]
(up to a subsequence). The equality

\[
\eta_{ij}^{(\epsilon)}=D_i\left((\epsilon^2+|\nabla u_\epsilon|^2)^{\frac{p+r}{2}-2}D_ju_{\epsilon}\right)
\]
follows from the pointwise convergence $\nabla u_{\epsilon m}\to \nabla u_\epsilon$.

For every $i,j=\overline{1,N}$ there exists a sequence $\{\eta_{ij}^{(\epsilon_k)}\}$ and $\eta_{ij}\in L^2(Q_T)$ such that $\eta_{ij}^{(\epsilon_k)}\rightharpoonup \eta_{ij}$. To identify $\eta_{ij}$ and prove \eqref{eq:main-2} (ii) we repeat \eqref{eq:sobolev-der}: for every $\phi\in C_c^{\infty}(Q_T)$

\[
\begin{split}
(\eta_{ij},\phi)_{2,Q_T} & =
\lim_{k\to \infty}\left(D_i\left((\epsilon_k^2+|\nabla u_{\epsilon_k}|^2)^{\frac{p+r}{4}-1}\right)D_ju_{\epsilon_k},\phi\right)_{2,Q_T}
\\
& = -\lim_{k\to \infty}\left((\epsilon_k^2+|\nabla u_{\epsilon_k}|^2)^{\frac{p+r-2}{4}}D_ju_{\epsilon_k},D_i\phi\right)_{2,Q_T}
\\
& = - \left( |\nabla u|^{\frac{p+r-2}{2}}D_ju,D_i\phi\right)_{2,Q_T}.
\end{split}
\]

\appendix

\section{Proof of Proposition \ref{int-by-parts-pro}}
\label{sec:intbyparts}
Let $\vec a$, $\vec b$ be given vectors. Integrating two times by parts we obtain
\begin{equation}
\label{eq:double-e}
\begin{split}
\int_\Omega & \operatorname{div}\vec a \operatorname{div}\vec b\,dx = \int_{\partial\Omega} a_{\nu} \operatorname{div}\vec b\,dS-\sum_{i=1}^N\int_\Omega a_i\sum_{j=1}^N D^2_{ij}b_j\,dx
\\
& = \int_{\partial\Omega} \left(a_{\nu} \operatorname{div}\vec b - \sum_{i,j=1}^{N}a_iD_ib_j \cos(\widehat{\vec \nu,x_j})\right)\,dS + \int_\Omega \sum_{i,j=1}^N D_j a_i D_i b_j\,dx
\\
& = \int_{\partial\Omega}\left(a_\nu \operatorname{div}\vec b- ((\vec a\cdot \nabla)\vec b)\cdot \vec \nu\right)\,dS + \int_\Omega \sum_{i,j=1}^N D_j a_i D_i b_j\,dx.
\end{split}
\end{equation}
At every point of $\partial\Omega$ we represent

\[
\begin{split}
& \vec a=\vec a_\tau +a_\nu\vec \nu,
\quad  \vec a_\tau=\sum_{j=1}^{N-1} a_j\vec \tau_j,\quad a_j=(\vec a,\vec \tau_j),\quad a_\nu=(\vec a,\vec \nu).
\end{split}
\]
Let $\vec c$ be a smooth vector. Using the representation $\vec c=\displaystyle\sum_{i=1}^{N-1}c_i\vec \tau_i+c_\nu\vec\nu$ and the formulas $(\vec \tau_i,\vec \nu)=0$, $\dfrac{\partial \vec \nu}{\partial \nu}\cdot \vec \nu=\dfrac{1}{2}\dfrac{\partial |\vec \nu|^2}{\partial \nu}=0$, we rewrite the divergence operator on $\partial\Omega$:
\[
\begin{split}
& \operatorname{div} \vec c=\sum_{i=1}^{N-1}\dfrac{\partial \vec c}{\partial s_i}\cdot\vec \tau_i +\dfrac{\partial \vec c}{\partial \nu}\cdot \vec \nu
= \sum_{i=1}^{N-1}\dfrac{\partial}{\partial s_i}\left(\sum_{j=1}^{N-1}c_j\vec \tau_j +c_\nu \vec \nu\right)\cdot\vec \tau_i +\dfrac{\partial}{\partial\nu}\left(
\sum_{j=1}^{N-1}c_j\vec \tau_j +c_\nu \vec \nu\right)\cdot \vec \nu
\\
& \qquad = \sum_{j=1}^{N-1}\dfrac{\partial c_j}{\partial s_j} + \sum_{i,j=1}
^{N-1}c_j \dfrac{\partial \vec\tau_j}{\partial s_i}\vec\tau_i + c_\nu\sum_{i=1}^{N-1}\dfrac{\partial \vec\nu}{\partial s_i}\vec\tau_i +\sum_{j=1}^{N-1}c_j\dfrac{\partial \vec\tau_j}{\partial \nu}\vec \nu +\dfrac{\partial c_\nu}{\partial\nu}.
\end{split}
\]
By the same token,

\[
\begin{split}
(\vec a\cdot \nabla)\vec b & = \sum_{j=1}^{N-1}a_j\dfrac{\partial \vec b}{\partial s_j} +a_\nu\dfrac{\partial \vec b}{\partial\nu}
=  \sum_{j=1}^{N-1}a_j\left(\sum_{k=1}^{N-1}\left(\dfrac{\partial b_k}{\partial s_j}\vec \tau_k + b_k\dfrac{\partial \vec \tau_k}{\partial s_j}\right)+ \dfrac{\partial b_\nu}{\partial s_j}\vec \nu +b_\nu \dfrac{\partial \vec \nu}{\partial s_j}\right)
\\
& \quad + a_\nu\sum_{k=1}^{N-1}\left(\dfrac{\partial b_k}{\partial \nu}\vec\tau_k+b_k\dfrac{\partial \vec\tau_k}{\partial \nu} +\dfrac{\partial b_\nu}{\partial\nu}\vec \nu+b_\nu\dfrac{\partial\vec\nu}{\partial \nu}\right)
\end{split}
\]
and

\[
\begin{split}
\left((\vec a\cdot \nabla)\vec b\right)\cdot \vec\nu & = \sum_{j,k=1}^{N-1}a_jb_k\dfrac{\partial \vec\tau_k}{\partial s_j}\vec \nu + \sum_{j=1}^{N-1}\dfrac{\partial b_\nu}{\partial s_j} +a_\nu\sum_{k=1}^{N-1}b_k\dfrac{\partial \vec\tau_k}{\partial \nu}\vec\nu +a_\nu\dfrac{\partial b_\nu}{\partial \nu}.
\end{split}
\]
Simplifying, we find that at every point of $\partial\Omega$

\[
\begin{split}
a_\nu \operatorname{div}\vec b - ((\vec a\cdot \nabla)\vec b)\cdot \vec \nu  & =  a_\nu\sum_{j=1}^{N-1}\frac{\partial b_j}{\partial s_j} + a_\nu \sum_{j,k=1}
^{N-1}b_k\dfrac{\partial \vec\tau_k}{\partial s_j}\vec \tau_j
+a_\nu b_\nu\sum_{j=1}^{N-1}\dfrac{\partial\vec \nu}{\partial s_j}\vec\tau_j-\sum_{j=1}^{N-1}a_j\dfrac{\partial b_\nu}{\partial s_j} -\sum_{j,k=1}^{N-1}a_jb_k\dfrac{\partial \vec\tau_k}{\partial s_j}\vec \nu.
\end{split}
\]
An auxiliary computation:

\[
\begin{split}
\operatorname{div}_\tau \left(a_\nu \vec b_\tau\right) & = \sum_{i=1}^{N-1}\dfrac{\partial}{\partial s_i}\left(\sum_{j=1}^{N-1}a_\nu b_j\vec\tau_j\right)\cdot \vec\tau_i
= \sum_{i,j=1}^{N-1}\dfrac{\partial a_\nu}{\partial s_i}b_j (\vec\tau_i,\vec\tau_j) + a_\nu\sum_{i,j=1}^{N-1}\dfrac{\partial b_j}{\partial s_i}(\vec\tau_i,\vec\tau_j) +a_\nu\sum_{i,j=1}^{N-1}b_j\dfrac{\partial\vec\tau_j}{\partial s_i}\vec\tau_i
\\
& = \sum_{i=1}^{N-1}\dfrac{\partial a_\nu}{\partial s_i}b_i  + a_\nu\sum_{i=1}^{N-1}\dfrac{\partial b_i}{\partial s_i} +a_\nu\sum_{i,j=1}^{N-1}b_j\dfrac{\partial\vec\tau_j}{\partial s_i}\vec\tau_i.
\end{split}
\]
A combination of the last two formulas gives

\[
\begin{split}
a_\nu \operatorname{div}\vec b & - ((\vec a\cdot \nabla)\vec b)\cdot \vec \nu =  \operatorname{div}_{\tau}\left(a_\nu \vec b_\tau\right) -\sum_{i=1}^{N-1}\left(\dfrac{\partial a_\nu}{\partial s_i}b_i + a_i\dfrac{\partial b_\nu}{\partial s_i}\right) +a_\nu b_\nu\sum_{j=1}^{N-1}\dfrac{\partial\vec \nu}{\partial s_j}\vec\tau_j -\sum_{j,k=1}^{N-1}a_jb_k\dfrac{\partial \vec\tau_k}{\partial s_j}\vec \nu.
\end{split}
\]
Fix a point $x_0\in \partial\Omega$, place the origin of the local coordinate system $\{y_i\}_{i=1}^N$ at $x_0$, and represent locally $\partial\Omega$ by the graph of the $C^2$-function: $y_N=\phi(y')$ where $y'=(y_1,\ldots,y_{N-1})$ belong to the tangent plane to $\partial\Omega$ at $y=0$. For every tangent to $\partial\Omega$ vectors $\vec \xi,\vec \eta$ the second fundamental form of $\partial\Omega$ at $x_0$ is defined by the equality
\[
\mathcal{B}(\vec\xi;\vec \eta)=-\dfrac{\partial \vec\nu}{\partial \xi}\cdot \eta=-\sum_{j,k=1}^{N-1}\dfrac{\partial \vec \nu}{\partial s_j}\vec \tau_k \xi_j\eta_k.
\]
On the other hand,

\[
\mathcal{B}(\vec \xi;\vec \eta)= \sum_{i,j=1}^{N-1}D^2_{y_iy_j}\phi(0)\xi_i\eta_j\qquad\text{and}\quad
-\sum_{i=1}^{N-1}\dfrac{\partial \vec\nu}{\partial s_i}\cdot \vec\tau_i=\operatorname{trace}\mathcal{B}=\sum_{i=1}^{N-1}D^2_{y_iy_i}\phi(0),
\]
because

\[
(\vec\tau_k,\vec \nu)=0\quad \Rightarrow\quad \dfrac{\partial \vec\tau_k}{\partial s_j}\cdot \vec \nu=-\vec\tau_k\cdot \dfrac{\partial \vec\nu}{\partial s_j}.
\]
Since the field $a_\nu\vec b_{\tau}$ is tangent to $\partial\Omega$, by the divergence theorem

\[
\int_{\partial\Omega}\operatorname{div}_{\tau}\left(a_\nu \vec b_\tau\right)\,dS=0.
\]
The integral over $\partial\Omega$ takes on the form
\begin{equation}
\label{eq:by-parts-a-b-e}
\begin{split}
\int_{\partial \Omega} & \left(a_\nu \operatorname{div}\vec b - ((\vec a\cdot \nabla)\vec b)\cdot \vec \nu\right)\,dS
\\
&
=  -\int_{\partial\Omega}\sum_{i=1}^{N-1}\left(\dfrac{\partial a_\nu}{\partial s_i}b_i + a_i\dfrac{\partial b_\nu}{\partial s_i}\right) \,dS - \int_{\partial\Omega}\mathcal{B}(\vec a_\tau;\vec b_\tau)\,dS-\int_{\partial\Omega}a_{\nu}b_{\nu}\operatorname{trace}\mathcal{B} \,dS
\\
& = -\int_{\partial\Omega} \left(\vec a_\tau \nabla _\tau (\vec b\cdot \vec\nu)+ \vec b_\tau \nabla _\tau (\vec a\cdot \vec\nu)\right)\,dS - \int_{\partial\Omega}\mathcal{B}(\vec a_\tau;\vec b_\tau)\,dS-\int_{\partial\Omega}a_{\nu}b_{\nu}\operatorname{trace}\mathcal{B} \,dS.
\end{split}
\end{equation}
Hence, the proof.

\bibliographystyle{siam}%{elsarticle-num}
\bibliography{arxivfinal}
\end{document}